\documentclass[10pt]{amsart} 

\usepackage{fullpage}

\usepackage{url}

\usepackage{amssymb}

\usepackage{cite}

\usepackage{blindtext}

\usepackage{graphicx}

\usepackage[usenames]{color}

\usepackage{accents}

\usepackage[normalem]{ulem}

\usepackage{hyperref}

\renewcommand{\a}{\alpha}
\renewcommand{\b}{\beta}
\newcommand{\g}{\gamma}
\renewcommand{\d}{\delta}

\newcommand{\e}{\varepsilon}
\newcommand{\f}{\varphi}
\newcommand{\s}{\sigma}
\newcommand{\Si}{\Sigma}
\renewcommand{\k}{\kappa}
\renewcommand{\l}{\lambda}

\renewcommand{\O}{\Omega}

\newcommand{\cA}{{\mathcal A}}
\newcommand{\cO}{{\mathcal O}}

\newcommand{\cC}{{\mathcal C}}

\newcommand{\cT}{{\mathcal T}}

\newcommand{\cB}{{\mathcal B}}
\newcommand{\cL}{{\mathcal L}}
\newcommand{\cE}{{\mathcal E}}

\newcommand{\cN}{{\mathcal N}}

\newcommand{\cH}{{\mathcal H}}
\newcommand{\cD}{{\mathcal D}}

\newcommand{\cZ}{\mathcal Z}
\newcommand{\cI}{\mathcal I}

\newcommand{\bR}{\mathbb R}

\newcommand{\bZ}{\mathbb Z}

\newcommand{\bE}{\mathbb E}

\newcommand{\bP}{\mathbb P}
 
\newcommand{\bV}{\mathbb V}

\newcommand{\Ric}{\mathrm{Ric}}
\newcommand{\Diff}{\mathrm{Diff}}

\newcommand{\tT}{\mathsf{T}}

\newcommand{\be}{\begin{equation}}
\newcommand{\ee}{\end{equation}}
\newcommand{\bes}{\begin{equation*}}
\newcommand{\ees}{\end{equation*}}

\newcommand{\tr}{\mathrm{tr}}

\newcommand{\beaa}{\begin{eqnarray*}}
\newcommand{\bea}{\begin{eqnarray}}
\newcommand{\beal}[1]{\begin{eqnarray}\label{#1}}
\newcommand{\bean}{\begin{eqnarray}\nonumber}
\newcommand{\beadl}[1]{\begin{deqarr}\label{#1}}
\newcommand{\eeadl}[1]{\arrlabel{#1}\end{deqarr}}
\newcommand{\eeal}[1]{\label{#1}\end{eqnarray}}
\newcommand{\eead}[1]{\end{deqarr}}
\newcommand{\eea}{\end{eqnarray}}
\newcommand{\eeaa}{\end{eqnarray*}}

\newcommand{\p}{\partial}

\renewcommand{\to}{\rightarrow}

\renewcommand{\phi}{\varphi}
\renewcommand{\epsilon}{\varepsilon}
\renewcommand{\hat}{\widehat}

\newcommand{\<}{\langle}
\renewcommand{\>}{\rangle}

\newcommand{\dm}{{\partial M}}

\newcommand{\w}{\widetilde}

\theoremstyle{plain}
\newtheorem{lemma}{Lemma}[section]
\newtheorem{proposition}[lemma]{Proposition}
\newtheorem{theorem}[lemma]{Theorem}

\newtheorem{corollary}[lemma]{Corollary}

\theoremstyle{remark}
\theoremstyle{definition}
\newtheorem{remark}[lemma]{Remark}

\def\blacksquare{\hbox to .60em {\vrule width .60em height .60em}}

\numberwithin{equation}{section}

\begin{document}

\title[ ]{Well-Posed Geometric Boundary data in General Relativity, III: \\
Conformal-mean curvature boundary data}

\author{Zhongshan An and Michael T. Anderson}

\address{Institute of Geometry and Physics, 
University of Science and Technology of China,
No. 99 Xiupu Road, Shanghai, China, 201305}
\email{zhshan.an@gmail.com}

\address{Department of Mathematics, 
Stony Brook University,
Stony Brook, NY 11794}
\email{michael.anderson@stonybrook.edu}

\begin{abstract} 
  This is the third work in a series on the (local in time) well-posedness of the initial boundary value problem (IBVP) for the vacuum Einstein equations in 
general relativity with geometric boundary conditions. Here we study the conformal-mean curvature boundary conditions, consisting of the conformal 
class of the boundary metric and mean curvature of the boundary. We prove that at metrics of uniformly bounded geometry to all orders, the linearized 
problem has a solution space with dense range in $C^{\infty}$ and establish a Holmgren-type uniqueness theorem valid for general smooth linearized solutions. 
These results require the addition of an arbitrary corner angle term at the intersection of the Cauchy surface and the timelike boundary. 
\end{abstract}


\maketitle

\section{Introduction}

   This work is the third in a series on the initial boundary value problem (IBVP) for the Einstein equations in general relativity. To describe the problem, let 
$M = I\times S$, where $S$ is an oriented $n$-manifold with boundary $\Si = \p S$ a compact $(n-1)$-manifold without boundary. The boundary surface 
need not be connected, but $S$ (and so $M$) are assumed to be connected. Consider Lorentz metrics $g$ on $M$ which are globally hyperbolic (in the sense of 
manifolds with timelike boundary) with time function $t: M \to I$ giving a parametrization of $I$ with $S = t^{-1}(0)$ serving as a Cauchy or initial data 
surface for $(M, g)$. The space of initial data $\cI$ is the space of all pairs $(\g, \k)$ where $\g$ is a Riemannian metric on $S$ and $\k$ is a symmetric 
bilinear form on $S$. The timelike boundary of $(M,g)$ is given by the cylinder $\cC = I\times \Si$. 

  Consider metrics $g$ as above which are vacuum Einstein metrics, i.e.
\be \label{Ein}
\Ric_g = 0.\footnote{All of the results of this work apply also to Einstein metrics with any cosmological constant $\Lambda$.}
\ee
Let $\bE$ be the space of all smooth ($C^{\infty}$) vacuum Einstein metrics $g$ on $M$; this is also known as the pre-moduli space or pre-covariant 
phase space of vacuum solutions. The issue is then to find effective descriptions of $\bE$ in terms of initial data on the initial slice $S$ and boundary data 
on the boundary $\cC$. 

Let then $\cI_0 \subset \cI$ be the subspace of all initial data $(\g, \k) \in \cI$ satisfying the vacuum constraint equations, cf.~\eqref{Gauss}-\eqref{GC} below.
Given a choice of boundary data space $\cB$, (local-in-time) well-posedness of the IBVP amounts to proving the existence of a 1-1 
correspondence (or better, homeomorphism) 
 \be \label{wp}
\cE \simeq \cI_0 \times_{c}\cB,
\ee
valid at least in suitable neighborhoods or thickenings of $S$ in $M$. Here 
$$\cE = \bE/\Diff_0(M)$$
is the quotient of $\bE$ by the action of the group $\Diff_0(M)$ of diffeomorphisms $\f: M \to M$ which restrict to the identity on $S \cup \cC$. The space 
$\cE$ is the natural moduli space of solutions, closely related to the covariant phase space of vacuum Einstein metrics, cf.~\cite{HW}. The space 
$\cB$ is the space of boundary data prescribed on the timelike boundary $\cC$, discussed in detail below. The subscript $c$ in \eqref{wp} 
denotes the compatibility conditions between the initial and boundary data at the corner $\Si$. 

\medskip 

  Thus the issue is whether one can effectively describe or parametrize the space of vacuum solutions on $M$ by their initial and boundary data (up to 
isometry in $\Diff_0(M)$). It is well-known that this is the case for the Cauchy or initial value problem where there is no boundary, cf.~\cite{C}, \cite{C2}, \cite{HE}.  
We note  that we are only concerned here with local-in-time well-posedness, asserting the existence and uniqueness of a solution $g$ with given initial and 
boundary data only for a finite, possibly small, proper time off the initial surface $S$. The long-time behavior of solutions is a fundamentally different 
issue, which to start requires a good understanding of the local-in-time behavior. We also assume here all data is $C^{\infty}$ smooth. 
 
    Let $g_S$ and $K_S$ denote the induced metric and second fundamental form (extrinsic curvature) of $g$ on $S$. Similarly, let $b(g) \in \cB$ 
 denote the boundary data induced by $g$. In addition, the correspondence above should be given by a mapping 
 $$\Psi: \cE \to \cI_0 \times_{c}\cB,$$
 $$\Psi(g) = (g_S, K_S, b(g)),$$
 which has at least a continuous inverse; this is the continuous dependence of the solution on initial and boundary data. In fact it is shown in 
 \cite{I} that the moduli space $\cE$ is a smooth Frechet manifold; in many cases the target data space $\cI_0 \times_{c}\cB$ is also a smooth 
 Frechet manifold and the map $\Psi$ is a smooth tame mapping of Frechet manifolds, (cf.~\cite{Ham} for background on Frechet manifolds). 
 
 \medskip 

   It is a longstanding open question whether there is a boundary data space $\cB$ for which such a parametrization $\Psi$ of $\cE$ exists at all. 
We consider only boundary data spaces $\cB$ which are geometric, i.e.~formed from a (half-dimensional) subspace of the space of 
Cauchy data, 
$$(\g_{\cC}, A),$$
at $\cC$. Here $\g_{\cC}$ is a Lorentz metric on $\cC$ (corresponding to the metric $g_{\cC}$ induced by $g$ on $\cC$) and $A$ is a symmetric bilinear
 form, (corresponding to the second fundamental form $A_{\cC}$ of $\cC$ in $(M, g)$). There has been a considerable amount of prior work on the 
 well-posedness of the IBVP, cf.~\cite{FN}, \cite{KRSW1}, \cite{KRSW2}, \cite{KW}, \cite{FS1}, \cite{FS2} and also ~\cite{ST} for  a general survey. 
 However, none of these prior studies concern the IBVP with geometric boundary data. Clearly, geometric boundary data are the most natural to 
 consider, from both geometric and physical viewpoints. 
 
    We also emphasize that for a geometric IBVP, the evolution of the metric $g$ off its initial data surface $S$ and the evolution of 
the boundary $\cC$ off its initial corner surface $\Si$ are both dynamical. Thus, the 'location' of the boundary is not known or fixed 
in advance. If \eqref{wp} holds, the evolution of the boundary $\cC$ off its initial corner $\Si$ is uniquely determined by the choice
 of initial and boundary data. Note also that generically, the vacuum metric $g$ will not extend past the boundary $\cC$ as a vacuum 
 solution (unless one can find very particular boundary conditions ensuring such a property). 
 
 \medskip 
  
 The simplest and perhaps most natural choices of boundary data are Dirichlet boundary data $\cB_{Dir}$, where one fixes or prescribes arbitrarily the 
boundary metric $\g_{\cC}$, or Neumann boundary data $\cB_{Neu}$, where one fixes the arbitrary second fundamental form $A$. These are the most 
common and well-behaved in many other geometric and physical PDE problems. However, it is shown in \cite{A2}, \cite{AA2} that these choices 
of boundary data are not well-posed for the Einstein equations in full generality, i.e.~without further restrictions.  On the other hand, it is proved in 
\cite{I} that Dirichlet boundary data is well-posed in a large open region of initial-boundary data. Further, in \cite{II} the IBVP is proved to be 
well-posed in general for a `twisting' of Dirichlet boundary data, which however is not fully geometric. Ideally, one would like to find a choice of 
fully geometric boundary data $\cB$ for which well-posedness holds in full generality. 

   In this work, we consider the boundary data space given by 
\be \label{cB}
\cB_C = {\rm Conf}^{\infty}(\cC)\times C^{\infty}(\cC),
\ee
$$b(g) = ([g_{\cC}], H_{\cC})$$
consisting of the pointwise conformal class of the metric induced on the boundary and a scalar function giving the mean curvature of the boundary. 

  The boundary conditions \eqref{cB} were first introduced in the Riemannian or Euclidean context \cite{A1}, where they were shown to give a 
well-posed elliptic BVP for the Einstein equations (in natural gauges), cf.~also \cite{W} for a physics-oriented perspective. In addition, they are 
well-posed in the parabolic or Ricci-flow setting, cf.~\cite{G}. Independently, they were also introduced in the analysis of the fluid-gravity 
correspondence relating the Navier-Stokes equations with the (Lorentzian) vacuum Einstein equations, cf.~\cite{BS}, \cite{AABN}, \cite{AGAM} and 
further references therein, cf.~also \cite{BSS}. In contrast to Dirichlet boundary data, it was shown in \cite{AA2} that the vacuum Gauss and Gauss-Codazzi 
constraint equations (equivalent to the vacuum Hamiltonian and momentum constraint equations) are always solvable along $\cC$, for arbitrary boundary 
data in $\cB_C$. This provides a preliminary but non-trivial check on the validity of well-posedness for the boundary data \eqref{cB}. 

\medskip 

  The choice of the gauge group $\Diff_0(M)$ plays an important role in understanding the relation between $\cE$ and $\cI_0\times_c \cB$. Consider 
for instance the intersection angle $\a_g$ along the corner $\Si$ between $S$ and $\cC$: 
\be \label{angle}
\a_g: \Si \to \bR,
\ee 
$$\a_g = g(\nu_S, \nu_{\cC}),$$
where $\nu_S$, $\nu_{\cC}$ are the future pointing and outward pointing unit normals to $S$ and $\cC$ respectively. Clearly $\a_g$ is 
$\Diff_0(M)$-invariant. It is shown in \cite{I, II} that $\a_g $ is determined by the initial and boundary data and their compatibility at $\Si$ for Dirichlet 
boundary data and twisted Dirichlet boiundary data respectively. If well-posedness as in \eqref{wp} holds, either the angle $\a$ must be determined by the compatibility 
conditions between initial and boundary data, or $\a$ is determined after the fact or aposteriori, only by the solution $(M, g)$ itself and thus is uniquely 
determined only indirectly by initial and boundary data. Conversely, the failure of well-posedness may imply that it is necessary to append, at least, an 
additional space of corner data to $\cI_0\times_c \cB$ to obtain well-posedness. 

  On the other hand, let $(\hat M, g)$ be a vacuum Einstein metric and suppose that $(\hat M, g)$ is (future and past) maximal globally hyperbolic with 
geometric boundary data $b(g)$ on $\cC$. Then $(\hat M, g)$ may be described as the maximal globally hyperbolic development of the initial, boundary 
and perhaps corner-type data of some Cauchy slice $S$. However, as is the case with the (pure) Cauchy problem \cite{C}, \cite{HE}, there are 
many choices of Cauchy surfaces giving the same (or isometric) maximal solution $(\hat M, g)$. Thus, it is natural to argue that the correct gauge 
group is $\Diff_0(\hat M)$ consisting of diffeomorphisms $\f: \hat M \to \hat M$ equal to the identity on the boundary $\cC$, thus dropping the restriction 
that $\f = {\rm  Id}$ on some Cauchy surface. With this larger gauge group, the corner angle may no longer be of any relevance. For example, in 
$2+1$ dimensions where solutions $(\hat M, g)$ are flat and so embed as domains in Minkowski space $\bR^{1,2}$, well-posedness under the larger 
gauge group $\Diff_0(\hat M)$ essentially only concerns the existence and unique determination of the evolution of the boundary cylinder 
$\cC \subset \bR^{1,2}$ off its initial surface $\Si \simeq S^1$ by the boundary data.

  However, an exact parametrization (analogous to \eqref{wp}) of the full space of solutions $\bE/\Diff_0(\hat M)$ is much more difficult to determine. 
Even in the case where there is no boundary, for instance when $S$ is compact without boundary, there is no known effective parametrization 
of $\bE/\Diff_0(\hat M)$. Solutions are uniquely determined by the space of vacuum solutions of the constraint equations on initial data, but the 
equivalence relation $(\g_1, \k_1) \sim (\g_2, \k_2)$ if and only if there are embeddings of this data into a common vacuum $(\hat M, g)$, is not at all 
well understood. 

  The same considerations as above apply to the gauge group of the covariant phase space $\bP$, cf.~\cite{HW}, \cite{LW}, \cite{WZ}, \cite{K}, 
where the gauge group is the group of diffeomorphisms $\Diff_{Sym}(M)$ of $M$ generated by vector fields in the kernel of the naturally defined 
pre-symplectic form $\O$ on $\bP$. This is the space and gauge group most relevant physically and to the study of quasi-local 
Hamiltonians and associated quasi-local conserved quantities, cf.~\cite{Sz} and references therein.  

\medskip 

   Returning to the boundary data \eqref{cB}, it turns out that it is in fact necessary to add on a space of corner data to the initial and 
boundary data for well-posedness, even under the large gauge group $\Diff_0(\hat M)$ (and even in $2+1$ dimensions). This follows by recent work of 
Anninos-Galante-Maneerat \cite{AGM1}, \cite{AGM2} and Liu-Santos-Wiseman \cite{LSW}, who carried out a general mode stability analysis for 
the linearization of the boundary value problem \eqref{cB} around certain special backgrounds. For instance, it is shown in \cite{LSW} that there is a 
1-parameter family of non-isometric timelike spherically symmetric domains (solid cylinders)  in flat Minkowski space $\bR^{1,3}$ which have the 
same boundary data $([\g], H) = ([\g_0], 2)$ as the standard round cylinder $\cC = \{r = 1\} \subset \bR^{1,3}$ with induced metric 
$\g_0 = -dt^2 + g_{S^2(1)}$. This family has the same fixed flat Cauchy surface $\{t = 0\}\cap \{r \leq 1\}$ in $\bR^{1,3}$; the 1-parameter family is 
described by varying the corner angle \eqref{angle} between $S$ and $\cC$. 

\medskip 

    Let $\cA = C^{\infty}(\Si)$ be the space of smooth functions on the corner $\Si$. Thus, in place of \eqref{wp}, one should consider the smooth map   
\be \label{Phi}
\Phi: \cE \to (\cI_0 \times \cB_C \times \cA)_c.
\ee
$$\Phi(g) = ((g_S, K_S), ([g_{\cC}], H_{\cC}), \a_g),$$
replacing the IBVP by the initial boundary corner value problem IBCVP. 

  Currently known methods to prove well-posedness require a choice of gauge to break the symmetry of the diffeomorphism group. 
Various choices of gauge may reduce the Einstein equations to a symmetric hyperbolic $1^{\rm st}$ order system or a hyperbolic system of 
$2^{\rm nd}$ order wave equations for instance. Given such a reduction, there is a well-established theory of the well-posedness of the IBVP 
for such systems provided the boundary data satisfy certain conditions, such as maximally dissipative boundary conditions, uniform KL condition, 
etc; we refer to \cite{BS}, \cite{Sa} and references therein  for further details. All prior work on the IBVP for the vacuum Einstein equations uses this 
approach. In the work of Friedrich-Nagy \cite{FN}, the corner angle $\a$ appears explicitly in the condition for well-posedness. However, in \cite{FN} 
the angle $\a$ is needed to precisely define the boundary conditions along $\cC$ and so the choice of the boundary data space itself depends on the corner 
angle. Similarly, the corner angle also appears in \cite{KRSW1} and related works, where again it is used to determine the boundary conditions. 

\medskip 

  In this work, we study the linearization or derivative of the map $\Phi$ in \eqref{Phi} in the smooth, $C^{\infty}$ setting. For the remainder of this work, 
we assume $\cB = \cB_C$. 

   To describe this in more detail, let $Met(M) = Met^{\infty}(M)$ be the space of $C^{\infty}$ smooth Lorentz metrics $g$ on $M$, as above globally 
hyperbolic with timelike boundary and let $S^2(M)$ be the space of smooth symmetric bilinear forms on $M$.  One has a natural smooth extension 
of the map $\Phi$ above to 
\be \label{Phi1}
\Phi: Met(M) \to [S^2(M) \times \cI \times \cB \times \cA]_c := \cT.
\ee
$$\Phi(g) = (\Ric_g, (g_S, K_S), ([g_{\cC}], H_{\cC}), \a_g).$$
Here $\cI$ denotes the full space of initial data $(\g, \k)$ on $S$, not necessarily satisfying any constraint or vacuum constraint equations. The target 
space $\cT$ is the full product space of free or uncoupled data, subject to the requisite compatibility conditions of the data along the corner $\Si$. 
As noted above, the initial, boundary and corner data are invariant under the action of the gauge group $\Diff_0(M)$; however in the bulk, the action 
$(\f, \Ric) \to \f^*(\Ric)$ is not invariant in general, unless one is on-shell, i.e.~on the space $\bE$ of vacuum solutions. 

  Closely related to the gauge group are the Hamiltonian and momentum constraint equations, i.e.~the Gauss and Gauss-Codazzi equations 
on $S$: 
\be \label{Gauss}
|K|^2 - H^2 - R_{g_S} = R_g - 2\Ric_g(\nu,\nu) = \tr_{g_S}\Ric_g - 3\Ric_g(\nu,\nu),
\ee
\be \label{GC} 
{\rm div}_{g_S}(K - Hg_S) = \Ric_g(\nu, \cdot). 
\ee 
Here $K = K_S, H = H_S = \tr_{g_S}K_S$ and $\nu = \nu_S$ are defined as above. The same equations hold along the timelike boundary $\cC$, 
but with $-R_{g_S}$ replaced by $+R_{g_{\cC}}$. These equations are defined on the domain $Met(M)$, but are not defined on the target 
space $\cT$, since the normal vector $\nu$ is not defined on $\cT$. For this reason, we enlarge the target domain by adding the normal 
vector $\nu_S$ along $S$ to the data. Let then $\hat \cT = (\cT \times \bV_S')_c$ where $\bV_S'$ is the space of smooth vector fields $\nu$ along $S$ 
nowhere tangent to $S$ and consider the extension $\hat \Phi$ of $\Phi$ in \eqref{Phi1} given by 
\be \label{Phi2}
\hat \Phi: Met(M) \to \hat \cT,
\ee
$$\hat \Phi(g) = (\Ric_g, (g_S, K_S, \nu_S), ([g_{\cC}], H_{\cC}), \a_g).$$ 
 Since the normal vector $\nu_S$ is part of the data in $\hat \cT$, the Gauss and Gauss-Codazzi equations \eqref{Gauss}-\eqref{GC} along 
$S$ now do make sense on $\hat \cT$; note the second equality in \eqref{Gauss}. (These equations are not defined along the boundary $\cC$ 
however). Denote then a general element in $\hat \cT$ as $(Q, (\g, \k, \nu), ([\s], \ell), \a)$ and let $\Lambda^1(S)$ be the 
space of 1-forms along $S$. There is a naturally defined constraint map 
\be \label{Const}
C: \hat \cT \to \Lambda^1(S),
\ee
$$C_0(Q, (\g, \k, \nu), ([\s], \ell), \a) = |\k|^2 - (\tr_{\g} \k)^2 - R_{\g} -  \tr_{\g} Q + 3Q(\nu,\nu),$$
$$C_i = {\rm div} (\k - (\tr_{\g}\k)\g) - Q(\nu, \cdot) .$$
The constraint equations \eqref{Gauss}-\eqref{GC} imply  ${\rm Im} \, \hat \Phi$ is contained in the zero-set $\cZ = C^{-1}(0)$ of $C$, 
so that 
\be \label{F2}
\hat \Phi: Met(M) \to \cZ \subset \hat \cT.
\ee
In particular, $\hat \Phi$ cannot be surjective onto $\hat \cT$. 
 
  The constraint space $\cZ$ involves a coupling of the bulk data and the initial data in the target space $\hat \cT$, (and hence induces further 
compatibility conditions along the corner $\Si$). However, again related to the gauge group $\Diff_0(M)$, there are also well-known restrictions 
on the bulk data $\Ric_g$ in $\hat \cT$. Let 
$$\Diff_1(M) \subset \Diff_0(M),$$
denote the subgroup of diffeomorphisms of $M$ which equal the identity to first order along $S$ and equal to the identity to zero 
order along $\cC$. Let $\b_g = \d_g + \frac{1}{2}d\tr_g$ be the Bianchi operator on symmetric bilinear forms on $M$. The Bianchi identity 
$\b_g \Ric_g = 0$ for Ricci curvature implies that $\Ric_g$ is not tangent to the orbit $\cO_{\Diff_1(M)}$ of the action of the diffeomorphism group 
$\Diff_1(M)$ on $\hat \cT$, 
$$(\Ric_g, \cdots)  \notin T(\cO_{\Diff_1(M)}),$$
where $\cdots$ denotes the remaining components of $\hat \Phi(g) \subset \hat \cT$ as in \eqref{Phi2}. We recall that $T(\cO_{\Diff_1(M)})$ is the 
space of symmetric forms of the form $\d^*X$ with $X = 0$ to first order on $S$ and zero order on $\cC$. In particular, $\hat \Phi$ cannot be 
surjective onto $\cZ$. 

\medskip 

  To deal with all of the issues mentioned above, a choice of gauge is needed to attack well-posedness problems. In this work, we use the well-known 
and commonly used harmonic gauge (also variously known as wave gauge, Bianchi gauge or de Donder gauge). Thus, assume $M \subset \bR^4$ 
topologically and consider the Euclidean metric $(\bR^4, g_{Eucl})$. (More generally, one may assume $M \subset \w M$ and let $g_R$ be a 
complete Riemannian metric on $\w M$). Let $V_g$ be the tension field of the identity wave map $(M, g) \to (\w M, g_R)$, cf.~\cite{GG}. In local 
coordinates $\{x^{\mu}\}$ on $M$, $V_g$ has the form 
$$V_g = \Box_g x^{\mu}\p_{x^{\mu}}.$$

  Much of this work is concerned with the analysis of the ``gauge fixed map" (related to the gauge-reduced Einstein equations):
\be \label{PhiH}
\Phi^H: Met(M) \to \cT^H,
\ee
$$\Phi^H(g) = (\Ric_g + \d_g^* V_g, (g_S, K_S, \nu_S, V_g|_{S}), ([g_{\cC}], H_{\cC}, V_g|_{\cC}), \a_g),$$
where $V_g|_S$ and $V_g|_{\cC}$ are the restrictions of the gauge field $V_g$ to $S$ and $\cC$ respectively. Here, $\cT^H = 
(\hat \cT \times \bV_M)_c$ is the space of free or unconstrained target data $\hat \cT \times \bV_M$, subject only to the smooth 
compatibility conditions along the corner $\Si$. The target space $\hat \cT$ for $\hat \Phi$ embeds in $\cT^H$ by setting $V = 0$ 
along $S\cup \cC$, (cf.~\S 5). 

   It is proved in \cite{I} that the gauged target space $\cT^H$ is a smooth Frechet manifold for Dirichlet boundary data $\cB_{Dir}$. It is 
easily verified that the same proof may be easily adjusted to prove that $\cT^H$ above is also a smooth Frechet manifold. In particular, 
it follows that $\Phi^H$ is a smooth tame map between Frechet manifolds. 

\medskip 

  To describe the first main result, let $\cD \subset Met(M)$ be the subspace of smooth metrics $g$ which have uniformly bounded geometry to 
all orders, i.e.~in suitable local coordinates with $g = g_{\a\b}dx^{\a}dx^{\b}$, 
$$||g_{\a\b}||_{C^m} \leq K,$$
for all $m \geq 0$ and for some fixed constant $K < \infty$. It is easy to see, cf.~\S 4, that $\cD$ is dense in $Met(M)$. 
Abusing notation slightly, we will use the same definition and notation for $h \in TMet(M)$ and for target data $\tau' \in T(\cT^H)$.
\footnote{All of the results to follow hold for the more general condition $||g_{\a\b}||_{C^m} \leq K^m$ for some $K$ and similarly for $\tau'$.}

\begin{theorem}\label{ThmI}
For any smooth vacuum Einstein metric $g \in \cD$, the derivative 
$$D\Phi_g^H: T_g Met(M) \to T(\cT^H),$$
$$D\Phi_g^H(h) =  \big([\Ric_g + \d_g^* V_g]'_h, (g_S, K_S, \nu_S, V_g|_{S})'_h, ([g_{\cC}], H_{\cC}, V_g|_{\cC})'_h, \a'_h),$$
is injective and has dense range. 

\end{theorem} 

  In fact we will show that $D\Phi_g^H$ surjects onto the dense space $\cD \subset T(\cT^H)$. The injectivity or uniqueness statement 
$$Ker D\Phi_g^H = 0$$
holds for general $C^{\infty}$ smooth $h$ and so may be considered as a Holmgren-type uniqueness result in this setting. 

  This leads to the following results for the (geometric) maps $\hat \Phi$ in \eqref{Phi2} and $\Phi$ in \eqref{Phi}, where the gauge 
fixing is removed step-by-step from $\Diff_1(M)$ to $\Diff_0(M)$. 

\begin{theorem}\label{ThmII} 
Suppose $\Ric_g = 0$ with $g \in \cD$. Then the map 
$$\hat \Phi: Met(M) \to \cZ,$$
has derivative $D\hat \Phi: TMet(M) \to T\cZ$, 
$$D\hat \Phi_g(h) = (\Ric'_h, (g_S, K_S, \nu_S)'_h, ([g_{\cC}], H_{\cC})'_h, \a'_h),$$
satisfying 
\be \label{transv}
{\rm Im}D \hat \Phi_g \oplus T(\cO_{\Diff_1 (M)}) \subset T \cZ,
\ee
with ${\rm Im}D \hat \Phi_g \oplus T(\cO_{\Diff_1 (M)})$ dense in $T \cZ$. 
The map $\hat \Phi$ thus descends to the quotient $Met(M)/\Diff_1(M)$ and the induced map 
$$D\hat \Phi: T(Met(M)/\Diff_1(M)) \to T(\cZ/\Diff_1(M)),$$
$$D \hat \Phi_g(h) = (\Ric'_h, (g_S, K_S, \nu_S)'_h, ([g_{\cC}], H_{\cC})'_h, \a'_h),$$
is injective with dense range. The equivalence class in $T(\cZ/\Diff_1(M))$ is represented by the symmetric forms $F$ on $M$ in the kernel 
of the Bianchi operator $\b_g$.

Further, dropping the normal vector $\nu_S$, the map $\hat \Phi$ descends to the full quotient $Met(M)/\Diff_0(M)$ and again the linearization 
$$D\Phi: T(Met(M)/\Diff_0(M)) \to T(\cZ/\Diff_0(M)),$$
$$D \Phi_g(h) = (\Ric'_h, (g_S, K_S)'_h, ([g_{\cC}], H_{\cC})'_h, \a'_h),$$
is injective with dense range. 
\end{theorem} 

  We expect that the results above are also true off-shell, where $g$ is no longer vacuum Einstein. In fact the dense range part of 
Theorem \ref{ThmI} is proved to hold off-shell. The proofs when off-shell require some modifications, but we will not address the details here. 

\medskip 

 Let $T_g \cE$ denote the tangent space to the Frechet manifold $\cE$ at $g$, equal to the linear space of solutions $h$ to the linearized vacuum 
Einstein equations $\Ric'_h = 0$ on $M$, modulo the equivalence relation $h \sim h + \d^*X$, where $X \in T(\Diff_0(M))$.  
Similarly, let $T(\cI_0)$ denote the space of infinitesimal variations $\iota' = (\g', \k')$ satisfying the linearization of the vacuum constraint equations. 

The following result is essentially an immediate consequence of the results above. 
\begin{corollary}\label{ThmIII}
For any $g \in \cD \cap \cE$, the derivative map 
\be \label{DPhi3}
D\Phi_g: T_g\cE \to T(\cI_0\times \cB \times \cA)_c,
\ee
is injective and has dense range. 
\end{corollary}

\medskip 

  The results above of course beg the question of whether $D \Phi_g$ in \eqref{DPhi3} is surjective in $C^{\infty}$ and so an isomorphism (and holds for 
general $g \in \cE$). The proofs of the results above show this would be the case if one had suitable apriori energy estimates for the domain data $h$ in 
terms of the target data $D\Phi_g(h)$. However, it follows from recent work in \cite{LRSW} that such energy estimates do not hold in general; in particular, 
they do not hold at the standard flat round cylinder $\cC = \{r = 1\} \subset \bR^{1,3}$ mentioned above. It is an open question whether energy estimates 
might nevertheless hold generically, i.e.~for generic $g \in \cE$. 

   Moreover, such surjectivity does not even hold in the Euclidean setting in general. Namely if $(M, g)$ is Euclidean Einstein and $(\p M, g_{\dm}) = 
 (S^n, g_{+1})$ is the round metric or conformal class on $S^n$, then for any variation $(0, H'_h)$ of the boundary data $([g_{\cC}], H_{\cC}) = ([g_{+1}], n)$ 
 in ${\rm Im} D\Phi$ one must have 
 $$\int_{S^n}X(H'_h) = 0,$$
 for all conformal Killing vector fields $X$ on $S^n$, cf.~\cite{A3}. Thus the image ${\rm Im} D\Phi$ is of codimension at least $n+1$ in the boundary data 
space $\cB_C$ and so of course not dense in $\cB_C$. Note this does not contradict the ellipticity of the boundary data in $\cB_C$; ellipticity implies 
only that $D\Phi$ is an isomorphism up a finite dimensional kernel and cokernel. 

\medskip 

   Following this introduction, the paper is organized as follows. In \S 2 we discuss background material needed for the work to follow. In \S 3, we develop 
the basic linearized theory for the problem, including the equations for the full (gauged) boundary data. This is used to prove the dense range property 
in Theorem \ref{ThmI} in \S 4, while \S 5 proves the uniqueness or injectivity statement in Theorem \ref{ThmI}. The proofs of the remaining 
results above are then given in \S 6.

 \medskip 
 
{\bf Acknowledgments:} We thank D. Anninos, D. Galante, C. Maneerat and E. Silverstein for interesting discussions related to this work. 
The work benefited greatly from participation in the conferences ``Timelike boundaries in theories of gravity" in Morelia, Mexico in 
August 2024 and ``Timelike boundaries in classical and quantum gravity" at the Simons Center for Geometry and Physics, Stony Brook, December 2025.

 \section{Background Material.}

  In this section, we discuss the basic background material needed for the work to follow. These include localization of the problem, 
the local geometry and compatibility conditions at the corner $\Si$ and basic information about the choice of gauge. 

\subsection{Localization.} 
 To prove the main results in \S 1, as usual with hyperbolic systems, we will localize the problem via a partition of unity and rescaling; 
this is the common ``frozen coefficient approximation" in the PDE literature. We then show later in \S 4 that these local solutions may be patched 
together to provide global (in space) solutions. 

  Given a metric $g$ on $M$ as in \S 1, the localization at a point $p \in \Si$ is a (small) neighbhorhood $U \subset M$ of $p$, diffeomorphic via 
a local chart to a Minkowski corner 
$$\mathbf R=\{(t=x^0,x^1,\dots ,x^n):t\geq 0, x^1\leq 0\}$$ 
with $S\cap U \subset \{t = 0\}$, $\cC \cap U \subset \{x^1 = 0\}$ and $x^{\a}(p) = 0$. Thus $t$ is a defining function for $S$ and $x^1$ is a 
defining function for $\cC$. Such local coordinates will be called an adapted local coordinate chart.  

As usual, Greek letters $\a, \b$ denote spacetime indices $0, \dots, n$, Roman letters denote spacelike indices $1, \dots, n$ and 
capital Roman letters $A,B$ denote corner indices $2, \dots, n$. 

  When $(M, g)$ is of size bounded away from $0$ and $\infty$, $U$ will be a small domain metrically, and correspondingly, the coordinates 
$x^{\a}$ will vary only over a small range. To renormalize this situation, as usual the metric and coordinates are rescaled simultaneously; thus 
for $\l$ small, set 
$$\w g = \l^{-2}g, \ \ \w x^\a=\l^{-1}x^\a$$
so that 
\be\label{rescale}
\w g(\p_{\w x^\a},\p_{\w x^\b})|_{\w x}=g(\p_{ x^\a},\p_{ x^\b})|_{\l x}. 
\ee
The equation \eqref{rescale} holds in the same way for any variation $h = \frac{d}{ds}(g + sh)|_{s=0}$ of $g$. Note that while the components 
$g_{\a\b}$ of $g$ are invariant under such a rescaling, all higher derivatives become small: 
$$\p_{\w x^{\mu}}^k \w g_{\a\b} = \l^k \p_{x^{\mu}}^k g_{\a\b}.$$ 
Thus the coefficients are close to constant functions in the rescaled chart, 
\be \label{eps}
||\w g - g_{\a_0}||_{C^{\infty}(U)} \leq \e = \e(\l, g),
\ee
where $g_{\a_0}$ is a flat (constant coefficent) Minkowski metric. This is the frozen coefficient approximation. 

   The exact form of the model flat metric $g_{\a_0}$ in adapted local coordinates is controlled by the corner angle $\a$ of $g$, 
or more precisely its value $\a_0 = \a(p)$ at $p$. In the $\w x^{\a}$ coordinates, the metric $\w g$ is $\e$-close to the flat model metric
\bes
g_{\a_0} =-dt^2-\a_0 dtdx^1+\sum_{i=1}^n(dx^i)^2.
\ees
Note that for $g_{\a_0}$, we have 
\be \label{normals}
\nu_S = - \nabla t = (1+\a_0^2)^{-1/2}(\p_0+\a_0 \p_1),  \ \ \nu_\cC=\nabla x^1 = (1+\a_0^2)^{-1/2}(\p_1-\a_0 \p_0),
\ee 
and 
$$g_{\a_0}(\nu_S,\nu_\cC)=\a_0.$$
Thus, given the defining function $t$ for $S$, the corner angle $\a_0$ determines the choice of the defining function $x^1$ for $\cC$ and thus 
the form of the abstract or coordinate-free Minkowski metric. 

  It is essentially clear that the local versions of the results stated in \S 1 hold for $g$ if and only if they hold for $\w g$. Briefly, 
write $D\Phi_g$ from \eqref{Phi1} in the form 
$$D\Phi_g(h) = (\Ric'_h, (h_S, K'_h)_S, (\ring{h}, H'_h)_{\cC}, (\a'_h)_{\Si}),$$
where $h_S = (g_S)'_h$, $K'_h = (K_S)'_h$, $\ring{h}$ is the trace-free part of $h^{\tT} = (g_{\cC})'_h$, and $H'_h = (H_{\cC})'_h$. 
Then the components of $D\Phi_g(h)$ and $D\Phi_{\w g}(\w h)$ are related in local coordinates by
$$(\Ric'_{\w g}(\w h), (\w h_S, \w K'_{\w h})_S, (\ring{\w h}, \w H'_{\w h})_{\cC}, \a'_{\w h})|_{\w x} = 
(\l^2 \Ric'_g(h), (h_S, \l K'_{h})_S, (\ring{ h}, \l H'_h)_{\cC}, \a'_h)|_{\l x}.$$ 
Similarly, $(\w \nu_S)'_{\w h} = (\nu_S)'_h$ and $\w V'_{\w h} = \l V'_h$ in local coordinates. 

  Now for any background metric $g$ and target data $\tau' \in T(\cT)$, form $\w g$ and $\w \tau'$ by the rescaling 
above, choosing $\l = \l(g)$ small enough, so that $\w g$ is $\e$-close to the constant coefficient metric $g_{\a_0}$. It is then easy to check 
that a solution to $D\Phi_{\w g}(\w h)=\w \tau'$ uniquely gives rise to a solution to $D\Phi_g(h)=\tau'$, where $h$ and $\w h$ 
are related as in \eqref{rescale}.\footnote{The same argument holds for an arbitrary cosmological constant $\Lambda$.} The same 
remarks apply to rescalings for $D\hat \Phi$ and $D \Phi^H$. 

\medskip 
 
  We will always assume that $U$ is embedded in a larger region $\w U$, so 
$$U \subset \w U,$$
with $\w U$ still covered by the adapted coordinates $(t, x^i)$, with $t \geq 0$ and $x^1 \leq 0$ in $\w U$ so that the initial surface $S$, boundary 
$\cC$ and corner $\Si$ in $\w U$ are an extension of the corresponding domains in $U$. We also assume \eqref{eps} still holds in $\w U$. 
All target data in $\cT$, $\hat \cT$ and later $\cT^H$ given in $U$ is extended off $U$ to be of compact support in $\w U$ away 
from $S\cap U$ and $\cC \cap U$. In particular, all target data vanishes in a neighborhood of the full timelike boundary of $\w U$ and in 
a neighborhood of the initial slice $\{t = 0\}$ away from $\cC \cap U$ and $S \cap U$ respectively. The same statements hold for variations of the 
target data, i.e.~in $T(\cT), T(\hat \cT)$ or $T(\cT^H)$. 

   For later reference, we note that the finite propagation speed property implies that solutions $h$ of the linear systems of wave equations on $\w U$ 
appearing in \S 3 and \S 4 then also have compact support in $\w U$ away from $S\cap U$ and $\cC \cap U$, for some definite (possibly small) time 
$t > 0$.

\subsection{Geometry at the corner.}
Next we discuss the form of the ambient metric $g$ at the corner $\Si$, determined by the compatibility conditions between $\cI$ and $\cB$ and 
the corner angle $\a$. 

Let $\tau = ((\g, \k, \nu), ([\s], \ell), \a)$ denote a general element in the target space $\cT$; thus $\tau$ represents general 
initial, boundary and corner data prescribing the geometric quantities $((g_S, K_S, \nu_S), ([g_{\cC}], H_{\cC}), \a_g)$. 

  The boundary condition $[g_{\cC}] = [\s]$ implies that there is a positive function $\f$  on $\cC$ such that 
\be \label{fconf}
g_{\cC} = \f^2 \s,
\ee
so only the conformal factor $\f$ is not determined by $[g_{\cC}]$. The compatibility between $g_S = \g$ and $g_{\cC} = \f^2 \s$ at $\Si$ 
immediately gives 
\be \label{C1}
\g |_{\Si} - \f^2 \s |_{\Si} = 0 ,
\ee
along $\Si$. Thus the conformal factor $\f$ is determined at $\Si$ by $\g$ and $\s$. Without loss of generality, we may assume that the induced 
volume forms $dv_{(\g_\Si)} = dv_{(\s_\Si)}$ agree, and hence 
\be \label{fS}
\f |_{\Si} = 1,
\ee
is determined at $\Si$. Thus, the full ambient metric $g = g_{\a\b}$ is determined at $\Si$ by initial and boundary data, except for the corner 
angle $\a = \<\nu_S, \nu_{\cC}\>$ in \eqref{angle}. It is clear that $\a$ is not determined by such data, and so must be added to the target data, as 
noted in \S 1. 

   Given an arbitrary choice of adapted local coordinates as in \S 2.1, the target metric $\s$ evaluated at $\Si$ has the 
general lapse-shift expression
\be \label{cm}
\s=-\rho^2 dt^2+w_Adtdx^A+\g_{AB},
\ee 
where $\g_{AB} = \g_{\Si}$. The scalar function $\rho$ and 1-form $w_A$ are thus given target data at $\Si$. The future-pointing  timelike unit 
normal vector of $\Sigma\subset (\cC,g_\cC)$ is then 
given by
\be \label{CShift}
T = \frac{\p_t-w^A\p_A}{\sqrt{\rho^2+|w|_\Si^2}}.
\ee
 Let $N=N^i\p_{x^i}$ denote the spacelike outward unit normal vector to the corner $\Si$ in $(S,g_S)$. Recall from \S 2.1 that the choice of 
 adapted coordinates $(t, x^i)$ requires specification of the corner angle $\a$ along $\Si$. Since $N$ is uniquely determined by the initial data 
 $g_S$, and since $x^1$ is a defining function for the corner $\Si \subset S$, $N^1\neq 0$. It is elementary to see that 
\be\label{cornerangle}
g(T,N) = - g(\nu_S, \nu_{\cC}) = -\a.
\ee

   In the $n+1$ or lapse-shift formalism, the spacetime metric $g$ evaluated on the initial surface $S$ is given by
\be\label{gcorner}
g=-\mu^2dt^2+X_idtdx^i+g_S.
\ee
Since $g_\cC$ is given by \eqref{cm} at $\Si$, by setting $x^1=0$ in the expression above, the compatibility condition at the corner 
$\Si$ gives 
\be \label{C1.5}
\mu=\rho,\ X_A=w_A \mbox{ on }\Si.
\ee 
It remains to calculate $X_1$ to fully determine $g_{\a\b}$ along the corner. Evaluating \eqref{gcorner} on the pair $(T, N)$ and using 
the equations \eqref{cornerangle} and \eqref{CShift} gives  
\bes
-\a=g(\tfrac{1}{\sqrt{\rho^2+|w|^2_\Si}}\p_t,N)
=\tfrac{1}{\sqrt{\rho^2+|w|^2_\Si}}X_i N^i.
\ees
Thus the component $X_1$ satisfies 
\bes
X_1=-\tfrac{1}{N^1}(\a\sqrt{\rho^2+|w|^2_\Si}+w_AN^A) \ \mbox{ on }\Si.
\ees
Since the full spacetime metric $g|_{\Si}$ is determined by the target data $\g, [\s], \a$ along the corner, the data $\nu$, which is used to prescribe 
the future-pointing timelike unit normal vector $\nu_S$, must satisfy the compatibility condition:
\be \label{C2}
\nu_S - \frac{1}{\sqrt{\mu^2 + \sum X_i^2}}(\p_t - X_i \p_i) = 0 \ \ {\rm on} \ \ \Si.
\ee
Here as always, $(t, x^i)$ are adapted local coordinates. The conditions  \eqref{C1}-\eqref{fS}, \eqref{C1.5}, and \eqref{C2} are the $C^0$ compatibility 
conditions for $\hat \cT$. 

\medskip

Regarding the first order behavior, the second fundamental form of $S$ is given by  
\be\label{pt-K}
\k = \frac{1}{2}\cL_{\nu_S} g|_S=\frac{1}{2}\cL_{\frac{\p_t-X}{\sqrt{\mu^2+|X|^2}}}g|_S=\frac{1}{2\sqrt{\mu^2+|X|^2}}(\cL_{\p_{t}}g|_S-\cL_{X}\g).
\ee
Hence $\p_t g_{ij}$, $(i, j =1, \dots, n)$ are determined by the initial data $(\g, \k)$, the boundary data $[\s]$ and the corner angle $\a$. 
 
 Next we observe that the initial velocity of the conformal factor $\f$, where $g_{\cC} = \f^2 \s$, is determined by the corner angle as well as the initial and 
boundary data. 

\begin{lemma} \label{cornerlemma}
One has the relation
\be \label{velocity}
(n-1)T(\f) = H_{\s} + \a H_{\Si} - \sqrt{1+\a^2}\tr_\Si K_S, 
\ee 
where $H_\s$ is the mean curvature of $\Si\subset (\cC,\s)$, $H_{\Si}$ is the mean curvature of $\Si \subset (S, \g)$ and 
$\tr_{\Si}K_S = \tr_{\g}\k$. In particular, all the terms on the right side of \eqref{velocity} are determined by the initial, boundary and 
corner data $(\g, \k), [\s], \a$. 
\end{lemma}

\begin{proof}

Let $e_A$, ($A = 2, \dots, n$) be a local orthonormal basis for $T\Si$. The well-known transformation rule for the Levi-Civita connection of 
conformally related metrics $g_\cC=\f^2\s$ states that 
\bes
\nabla^{g_{\cC}}_{e_A}T = \nabla^\s_{e_A}T + e_A(\ln \f)T - T(\ln \f)e_A,
\ees
and hence on $\Si$,  
\bes
\begin{split}
g(\nabla_{e_A}T, e_A) = \s(\nabla^\s_{e_A}T + e_A(\ln \f)T - T(\ln \f)e_A, e_A) = H_\s - (n-1)T(\f) .
\end{split}
\ees
On the other hand $T = \sqrt{1+\a^2}\nu_S-\a N$, so that 
\bes
\begin{split}
g(\nabla_{e_A}T,e_A)=g(\nabla_{e_A}[\sqrt{1+\a^2}\nu_S-\a N], e_A)=\sqrt{1+\a^2}\tr_\Si K_S - \a H_\Si,
\end{split}
\ees
which gives \eqref{velocity}. 
\end{proof}

The linearization of \eqref{velocity} will be important in the linear analysis in \S 3, cf.~Lemma \ref{u}. Note that the mean curvature $H_{\cC}$ does not enter 
the expression \eqref{velocity}.

Furthermore, a similar calculation can be used to determine the $t$-derivative of $\s$ at $\Si$. Namely, on restriction to $\Si$, we have  
$$\cL_{T^{\cC}}(\f^2 \s) = \cL_{T^{\cC}}g = \cL_{(\sqrt{1+\a^2}\nu_S - \a N)}g = \sqrt{1+\a^2}\k - \a B_{\Si},$$
where $B_{\Si}$ is the second fundamental form of $\Si \subset (S,\g)$. Since $\cL_{T}(\f^2 \s) = 2\f T(\f)\s + \f^2\cL_{T}\s$, it 
follows that, on $\Si$, 
\be \label{C3}
\cL_{T} \s + 2T(\f) \s - \sqrt{1+\a^2}\k  + \a B_{\Si} = 0,
\ee 
where $T(\f)$ is given by \eqref{velocity}. This is the first order compatibility condition on the data $(\g,\k), [\s], \a$.

   We summarize the analysis above in the following Lemma. 
\begin{lemma}\label{corner}
The 0 and 1 jets of $g_{\a\b}$ are all determined by the initial, boundary and corner data of $g$ in $\cT$ except for the components 
$\p_t g_{01}, \p_1 g_{0\alpha}$ on $\Si$.
\end{lemma} 

  The same result holds for infinitesimal deformations $h$ of $g$. Lemma \ref{corner} does not depend on any choice of gauge or 
adapted local coordinates for $g$. Geometrically, the terms $\p_1 g_{0i}$, $i = 0,2, \dots, n$ correspond to the component $A(T, \cdot)$ 
of the $2^{\rm nd}$ fundamental form $A = A_{\cC}$ of $\cC$ in $(M, g)$. The remaining two terms $\p_t g_{01}$ and $\p_1 g_{01}$ correspond 
to the time and radial variation of the corner angle (in adapted coordinates). We note that some of these remaining terms are determined on 
$\Si$ by a choice of normal vector $\nu_S$ and gauge $V$; cf.~Lemma \ref{hS}.

\subsection{Harmonic gauge.}

To prove well-posedness or solvability, it is well-known that it is necessary to work with a specific choice of gauge field or local slice to 
the action of the diffeomorphism group $\Diff_0(M)$ acting on the space of solutions. Such a choice determines, either directly or indirectly, 
a choice of space-time or lapse-shift decomposition of the ambient or bulk metric. In addition to the choice of gauge field in the bulk, the choice of 
its boundary conditions at the boundary $\cC$, in relation to the boundary conditions for the metric $g$, plays an important role in the well-posedness 
problem.  

As usual with the Cauchy problem for the Einstein equations, we work below with the (generalized) harmonic or wave coordinate gauge, described 
locally by the coordinate vector field    
\be \label{harm}
V = V_g = (\Box_g x^{\mu})\p_{x^{\mu}},
\ee
where $x^{\mu}$ are (fixed) adapted local coordinates, so $t = x^0$ and $x^1$ are defining functions for $S$ and $\cC$ respectively. 
We emphasize that, given a choice of adapted local coordinates, the gauge field $V_g$ is uniquely determined by $g$. 
We then form the harmonic ``gauge reduced map" $\Phi^H$ by considering, as in \eqref{PhiH}  
\be \label{PhiH1}
\Phi^H: Met(M) \to \cT^H,
\ee
$$\Phi^H(g) = (\Ric_g + \d_g^* V_g, (g_S, K_g, \nu_S, V_g|_S), ([g_{\cC}], H_{\cC}, V_g|_{\cC}), \a_g).$$ 
Note that we are thus imposing Dirichlet boundary conditions for the gauge field on $\cC$, (in addition to Dirichlet initial 
conditions on $S$). One might consider other boundary conditions for $V_g$, but we will not pursue that here. 
  
  As mentioned in the Introduction, we note that it is straightforward to globalize \eqref{harm} by replacing the local 
expression for $V$ by a wave map to Riemannian target space. The proof that the choice of $V$ does lead to a suitable choice of gauge is 
given in Remark \ref{goodg}. 

\medskip 

  The second and higher order compatibility conditions at the corner involve coupling of the bulk term $Q = \Ric_g$ of the target 
space with initial and boundary data and are derived by the same process as above. For the linear theory discussed here, it will not be 
necessary to specify these higher order conditions along $\Si$ explicitly. We refer then to the first part \cite{I} of this series for the 
detailed form of the higher order compatibility conditions, where they are needed for the nonlinear theory. (These are done in \cite{II} 
for Dirichlet boundary conditions, but there are only minor differences with the $\cB_C$ boundary conditions). As a consequence of 
that analysis, it follows that $\cT^H$ is a smooth Frechet manifold, as is $Met(M)$ (and also $\cE$ as shown in \cite{II}). 

   We thus have the three maps $\Phi$, $\hat \Phi$ and $\Phi^H$ in \eqref{Phi}, \eqref{Phi2}, \eqref{PhiH} or \eqref{PhiH1}. Most all of the analysis 
to follow is concerned with the gauge-fixed map $\Phi^H$. After developing a detailed understanding of the linearization of $\Phi^H$, 
it is then quite straightforward to derive a similar understanding for the geometric maps $\hat \Phi$ and $\Phi$.

\subsection{Function spaces.}

  We with work with the function spaces naturally associated to energy estimates for wave-type equations. Thus let $D$ denote a domain given as either $M$, $S$, 
$\cC$, $\Si$ or the corresponding $t$-level sets $S_t$, $\Si_t$, or the corresponding $t$-sublevel sets, $M_t$, $\cC_t$; ($M_t = \{p \in M: t(p) \leq t\}$). Define the 
Sobolev $H^s$ norm on functions on $D$ by 
 $$||v||_{H^s(D)}^2 = \sum_{k=0}^s \int_{D} |\p_D^k v|^2dv_D,$$
where $\p_D^s$ consists of all partial coordinate derivatives tangent to $D$ of order $\leq s$. In place of coordinate derivatives, one may use the 
components of covariant derivatives of $v$ up to order $s$, given a (background) Riemannian metric $g_D$ on $D$. The volume form $dv_D$ is also 
induced from the metric $g_{D}$.  

  Define the stronger $\bar H^s$ norm by including all space-time derivatives up to order $s$, so that 
\be \label{barnorm}  
||v||_{\bar H^s(D)}^2 = \sum_{k=0}^s \int_{D} |\p_M^k v|^2dv_D,
\ee
where $\p_M$ denotes partial derivatives along all coordinates of $M$ at $D$. 
 
  Finally, for the Cauchy slices $S_t$, define the boundary stable $H^s$ norm on $S_t$ by 
\be \label{bsnorm}
||v||_{\bar \cH^s(S_t)}^2 = ||v||_{\bar H^s(S_t)}^2 + ||v||_{\bar H^s(\cC_t)}^2.
\ee
  
  We will use the function space  
\be \label{N}
\cN^s(M) = C^0(I, \bar \cH^{s}(S)),
\ee
for functions on $M = M_1$. With the usual $C^0$ norm in $I$, the space $\cN^s$, with associated norm $|| \cdot ||_{\cN^s}$ is a 
separable Banach space. Norms such as \eqref{N} are commonly used for the space of solutions of wave-type equations, 
particularly regarding the Cauchy problem. Note however the use of the weaker $\bar H^s$ norm \eqref{barnorm} in \eqref{N} in 
place of the stronger boundary stable norm $\w \cH^s$ in \eqref{bsnorm}. 

 It is easy to see that 
$$H^{s+1}(M) \subset \cN^s(M) \subset H^s(M).$$
We will always assume $s$ is large, and in particular $s > \frac{n+1}{2} + 2$, so that the components of $g$ are at least $C^2$.

 \section{Linearized Equations} 
 
   In this section, we begin the analysis of the derivative $D\Phi^H$ of the map $\Phi^H$ in \eqref{PhiH1}. 
\be \label{DPhiH}
D\Phi^H: T(Met(M)) \to T(\cT^H),
\ee
$$D\Phi_g^H(h) = ((\Ric_g + \d_g^* V_g)'_h, (h_S, K'_h, (\nu_S)'_h, V'_h), (([g_{\cC}]'_h, H'_{h}, V'_h), \a'_h).$$    
We will use the discussion here to analyse the surjectivity and injectivity of $D\Phi^H$ in \S 4 and \S 5. 
  
   All of the results of this section are valid either globally, for $M$ as in \S 1, or locally, for localized corner neighborhoods $U \subset \w U$ as 
described in \S 2.1. In the localized setting, to simplify notation, we will always denote $\w g$ simply by $g$. We assume all data is $C^{\infty}$ smooth. 

  For later purposes, we recall that in the local setting of \S 2.1, all derivatives $\p g$ of $g$ in adapted local coordinates, in particular the 
$2^{\rm nd}$ fundamental form $A$ of $\cC$ in $(M, g)$, are $O(\e)$ in $C^{\infty}$. Further, again as discussed in \S 2.1, all variations 
$\tau'$ of the target data in $\hat \cT$ or $\cT^H$ are assumed to have compact support in $\w U$ away from $S\cap U$ and $\cC \cap U$.
 
 \subsection{Bulk equations.}

To begin, consider the bulk term in \eqref{DPhiH}, i.e.~the equation
 \be \label{L}
L(h) = (\Ric_g + \d_g^*V_g)'_h = F. 
 \ee
 The first term $\Ric'_h$ has the form, cf.~\cite{B}, 
$$\Ric'_h = \tfrac{1}{2}D^*D h +\Ric \circ h - Rm(h) - \d^*\b h.$$
Here and in the following, all geometric tensors and operators are with respect to $g$, so we omit the subscript $g$. 
Recall also that $\b$ is the Bianchi operator. Since $\d^* V = \frac{1}{2}\cL_V g$, $2(\d^*V)'_h = \cL_V h + 2\d^* V'_h  = \nabla_V h + 
2\d^* V \circ h + 2\d^*V'_h$. Also $V = \Box x^{\a} \p_{\a}$ so that $V'_h = (\Box)'_h x^{\a}\p_{\a}$ and a standard formula, (cf.~\cite{B}), 
gives $(\Box)'_h x^{\a} = -\<D^2 x^{\a}, h\> + \<\b h, dx^{\a}\>$, so that 
 \be \label{V'h}
 V'_h = \b h - \<D^2 x^{\cdot}, h\>\p_{x^{\cdot}}. 
 \ee
 The first term here gives rise to a cancellation with the term $-\d^*\b h$ above, which then gives 
$$(\Ric_g + \d^*V)'_h = \tfrac{1}{2}D^*D h + \Ric \circ h - Rm(h) + \tfrac{1}{2}\nabla_V h + \d^* V \circ h - \d^* (\<D^2 x^{\a}, h\> \p_{\a}) .$$
The leading order term here is the tensor wave operator $D^*D$; the zero order terms in $h$ have coefficients involving two derivatives of 
$g$ while the first order terms in $h$ have coefficients involving one derivative of $g$. The operator $D^*D$ operating on symmetric bilinear 
forms $h$ may be written as a coupled system of scalar equations for the components $h_{\a\b}$ of $h$ of the form 
$$(D^*Dh)_{\a\b} = -\Box_g h_{\a\b} + S_{\a\b}(h),$$
where $S(h)$ has lower order terms as above. We thus write the linearization \eqref{L} in the form 
\be \label{L2}
(L(h))_{\a\b} = -\tfrac{1}{2}\Box_g h_{\a\b} + P_{\a\b}(\p h) = F_{\a\b},
\ee
where $P$ is the collection of lower order - zero and first order - terms in $h$. The first order operator $P$ is linear in $h$, 
but is highly coupled. In the localized setting of \S 2.1, $P$ is a small perturbation term; by \eqref{eps}, the coefficients of $P$ are all $O(\e)$.

\medskip 

We note the following Lemma for the initial data for \eqref{L2} for future reference. 

\begin{lemma}\label{hS}
In adapted local coordinates along $\Si$, the initial data 
$$(h_{\a\b}, \p_t h_{\a\b}) \  {\rm on} \  S,$$
are uniquely determined by the target initial data $(h_S, K'_h, (\nu_S)'_h, V'_h)$, and corner data $\a'_h$ in \eqref{DPhiH}. The converse also holds. 
\end{lemma}

\begin{proof} 
We show first that the nonlinear version of this result holds, so with $g$ in place of its variation $h$. First, the determination of the corner angle 
$\a$ along $\Si$ determines a choice of adapted local coordinates, up to the natural equivalence relation, cf.~\S 2.1. Having then fixed the 
coordinates through $\a$, the target initial data for $g$ along $S$ are then $(g_S, K_S, \nu_S, V|_S)$. In such local coordinates, $g_S$ determines 
$g_{ij}$, $i, j = 1,\dots, n$. Also, since $\nu_S$ and the coordinate $t$ are given, the lapse-shift $(\mu, X)$ with $\nu_S = \mu \p_t + X$ are given, 
and so the components $g_{0\a}$ are also determined. Similarly, $K_S$ then determines $\p_t g_{ij}$. Finally, it is easy to see that from these prior 
determinations, $(V_S)^{\a} = \Box_g x^{\a}|_S$ determines $\p_t g_{0\a}$ on $S$. 

  We may apply the arguments above to the linearization $h_{\a\b} = \frac{d}{ds}(g + sh)_{\a_s\b_s}$. The linearization is a sum of two terms. 
The first is the linearization $h$ of $g$ in fixed coordinates and the second is given by the components of $g$ in the infinitesimal 
variation of the coordinates (the variation of the change of basis). The sum these two terms is determined by the argument above, while the 
variation of the coordinates is determined up to equivalence by $\a'$. 

\end{proof}

\subsection{Gauge equations.}

 We next discuss the gauge field $V = V_g$ and its linearization $V'_h$ at $g$. Let 
\be \label{V0}
\Ric_g + \d^*V = Q,
\ee
so $Q$ is determined by the target data in $\cT^H$.  

\begin{lemma}\label{Gauge-lemma}
Let $\Box_g$ denote the wave operator $\Box_g = - D^*D$ acting on vector fields $V$ on $M$. Then  
\be \label{V}
-\tfrac{1}{2}[\Box_g + \Ric_g](V) = \b_g Q.
\ee
The initial data $V|_S$, $(\p_t V)|_S$ and boundary data $V|_{\cC}$ are determined by the target data  
$$(\Ric + \d^*V, (g_S, K_S, \nu_S, V|_S)_S, V|_{\cC})$$
in $\cT^H$ and hence $V$ on $M$ is uniquely determined by target data. 

\end{lemma}

\begin{proof}
Applying the Bianchi operator $\b_g = \d_g + \frac{1}{2}d \tr_g$ to \eqref{V0} gives 
\be \label{V2}
\b_g \Ric_g + \b_g \d^*V = \b_g Q,
\ee
which, by the Bianchi identity $\b_g \Ric_g = 0$, gives \eqref{V} via a standard Weitzenbock formula 

The Dirichlet data for $V$ along $S$ and $\cC$ are target data in $\cT^H$. We claim that the $t$-derivative $\p_t V$ on $S$ is also 
determined by target data. This follows from the Gauss and Gauss-Codazzi identities \eqref{Gauss}-\eqref{GC}. For $\nu = \nu_S$ the 
unit timelike normal to $S$, these identities show that the form $E(\nu, \cdot) = \Ric(\nu, \cdot) - \frac{1}{2}R g(\nu, \cdot)$ is determined by 
initial data $\iota = (g_S, K_S)$ along $S$. By \eqref{V0}, it follows that $\d^*V(\nu, \cdot) - \frac{1}{2}\tr \d^*V g(\nu, \cdot)$ is determined by 
initial data. Since $V$ is determined along $S$ by the target data, it follows easily that $\nabla_{\nu}V$ and hence $\p_t V$ is determined 
along $S$ by target data. It follows from the standard existence and uniqueness of solutions to the wave equation \eqref{V} with given initial 
and Dirichlet boundary data that $V$ is uniquely determined by target data. 

\end{proof}

 A similar result holds for the linearization $V'_h$ of the gauge field $V = V_g$. 
\begin{lemma}\label{Gauge-lemma2}
We have 
\be \label{V'}
-\tfrac{1}{2}(\Box_g + \Ric_g)V'_h + \tfrac{1}{2}\nabla_V V'_h = \b_g F + \b'_h \Ric_g + O_{2,1}(V,h),
\ee
where $O_{2,1}(V,h)$ is $2^{\rm nd}$ order in $V$ and $1^{\rm st}$ order in $h$. Moreover, $O_{2,1}(V,h) = 0$ if $V = 0$. 

  The initial data $(V'_h)|_S$, $(\p_t V'_h)|_S$ and boundary data $V'_h|_{\cC}$ are determined by the target data  
$$\big((Ric + \d^*V)'_h, (h_S, K'_h, (\nu_S)'_h, V'_h)_S, (V'_h)_{\cC}\big)$$
in $T(\cT^H)$.
\end{lemma}

\begin{proof}
Again applying the Bianchi operator to both sides of \eqref{L} gives:
\bes
\b_g \Ric'_h+\b_g \d^* V'_h+\b_g[(\d^*)'_h V]=\b_g F
\ees
which via the Weitzenbock formula as before implies 
\be \label{V'1}
-\tfrac{1}{2}[\Box V'_h+\Ric_g(V'_h)]=\b_g F + \b'_h \Ric_g -\b_g[(\d^*)'_h V].
\ee
Simple calculation gives $(\d^*)'_hV = \frac{1}{2}\nabla_V h + \d^*V\circ h$, so that 
 $$\b[(\d^*)'_hV] = \tfrac{1}{2}\b(\nabla_Vh) + O_{2,1}(V,h) = \tfrac{1}{2}\nabla_V \b h + O_{2,1}(V,h) = \tfrac{1}{2}\nabla_V V'_h + O_{2,1}(V,h),$$
where we have used \eqref{V'h} in the last equality. This gives \eqref{V'} 

  As in Lemma \ref{Gauge-lemma}, the Dirichlet data for $V'_h$ along $S$ and $\cC$ are given as target space data and we use the 
constraint equations \eqref{Gauss}-\eqref{GC} to determine the initial velocity $\p_t V'_h$. As before, the bulk equation yields:
\bes
\begin{split}
&(\Ric-\tfrac{1}{2}Rg)'_h(\nu_S)+[\d^*V-\tfrac{1}{2}({\rm div}V)g]'_h(\nu_S)\\
&=[F-\tfrac{1}{2}(\tr F)g](\nu_S)+\tfrac{1}{2}\<h,\Ric_g+\d^* V_g\>g(\nu_S)-
\tfrac{1}{2}\tr(\Ric_g+\d^* V_g)h(\nu_S)
\end{split}\ees
and thus
\bes\begin{split}
&[\d^*V-\tfrac{1}{2}({\rm div}V)g]'_h(\nu_S)\\
&=-[(\Ric-\tfrac{1}{2}Rg)(\nu_S)]'_h+(\Ric-\tfrac{1}{2}Rg)((\nu_S)'_h)\\
&+[F-\tfrac{1}{2}(\tr F)g](\nu_S)+\tfrac{1}{2}\<h,\Ric_g+\d^* V_g\>g(\nu_S)-
\tfrac{1}{2}\tr(\Ric_g+\d^* V_g)h(\nu_S)
\end{split}\ees
on $S$.
By the constraint equations \eqref{Gauss}-\eqref{GC}, $[(\Ric-\tfrac{1}{2}Rg)(\nu_S)]'_h$ is given by 
\bes
\big(-\tfrac{1}{2}[|K_S|^2-(\tr_{g_S}K_S)^2+R_S]'_{(h_S,K'_h)_S},~[{\rm div}_{g_S}K_S-d_S(\tr_{g_S}K_S)]'_{(h_S,K'_h)_S}\big).
\ees
Thus the target data in $T(\cT^H)$ uniquely determine the vector field 
\bes
[\d^*V'_h](\nu_S)-\tfrac{1}{2}({\rm div}V'_h)g(\nu_S),
\ees
along $S$. Since the initial data of $V'_h$ is already determined, this uniquely determines the vector field 
\be
\nabla_{\nu_S} V'_h \ \ {\rm along} \ \ S.
\ee

\end{proof}

\begin{remark}\label{goodg}
{\rm Lemma \ref{Gauge-lemma} implies the standard result that solutions of the gauge-reduced Einstein equations 
$$\Ric_g + \d^*V = 0,$$
with target data $V = 0$ at $S\cup \cC$ satisfy $V = 0$ on $M$, and so are solutions of the vacuum Einstein equations 
$$\Ric_g = 0$$
on $M$. A similar result holds locally, under the compact support conditions as discussed in \S 2.1.  

  In the case that $\Ric_g = V = 0$, by Lemma \ref{Gauge-lemma2} the same result and proof as above holds for the linearized gauge 
  $V'_h$.   
  
  Similarly, we note that any metric $g$ can locally be brought into harmonic gauge by means of a diffeomorphism 
 $\f \in \Diff_0(\w U)$. To see this, given $g$ and an adapted local chart $x^{\a}$ for $g$, let $y^{\a}$ be harmonic or wave coordinates with 
 the same initial and boundary data as the given coordinates $x^{\a}$, 
 \be \label{ya}
 \Box_g y^{\a} = 0.
 \ee
 Since Dirichlet data are well-posed for the scalar wave equation, \eqref{ya} has a unique solution. The coordinate change 
 $y^{\a} = y^{\a}(x^{\b})$ defines a diffeomorphism $\f \in \Diff_0(\w U)$. Setting $\w g = \f^*g$, one has 
 $$\Box_{\w g}x^{\a} = \Box_g y^{\a} = 0.$$
 
   The same result holds in the linearized setting. 

}
\end{remark}

\subsection{Boundary equations.}

   The geometric boundary data for $h$ on $\cC$ are given by 
$$(\ring{h}, H'_h),$$
where $\ring{h} = [g_{\cC}]'_h$ denotes the variation of the conformal class $[g_{\cC}]$ and $H'_h$ is the variation of the mean curvature $H_{\cC} := H_g$ 
of $\cC$ in $(M, g)$. Let $h^{\tT}$ denote the restriction of $h$ to $\cC$, so $h^{\tT}$ gives the variation of the boundary metric $g_{\cC}$. Then $\ring{h}$ 
will be identified with the trace-free part of $h^{\tT}$, so 
$$\tr_{g_{\cC}}\ring{h} = 0,$$
giving $h^{\tT} = \ring{h} + (\frac{1}{n}\tr_{g_{\cC}}h^{\tT}) g_{\cC}$. 

  While $\ring{h}$ represents Dirichlet boundary data, $H'_h$ represents Neumann or a Dirichlet-Neumann mixture of Cauchy data at $\cC$. 
The undetermined Dirichlet data for $h$ at $\cC$ are thus the terms 
\be \label{uh}
u = \tfrac{1}{n}\tr_{\cC}h \ \ {\rm and} \ \ h(\nu_{\cC}, \cdot),
\ee
corresponding to the variation of the conformal factor and the variation of the normal vector $\nu_{\cC}$ at $\cC$ respectively. 
We note that the pair 
\be \label{conj}
(u, H'_h),
\ee
or their non-linear analogs $(\f, H)$ for $\f$ as in \eqref{fconf} are a canonical conjugate pair (up to a multiplicative constant) for the 
Einstein-Hilbert action with a natural boundary action, cf.~\eqref{EH} below. The term $h(\nu_{\cC}, \cdot)$ is gauge dependent, so should 
naturally be determined, or at least depend on the choice of gauge boundary condition at $\cC$. 

  By definition, 
 \be \label{nuC}
\nu_{\cC} = \frac{1}{|\nabla x^1|}\nabla x^1 =  \frac{1}{|\nabla x^1|}g^{\a1}\p_{\a}.
\ee
The expression \eqref{nuC} defines the extension of $\nu_{\cC}$ at $\cC$ into a neighborhood of $\cC$ in $(M, g)$. 
In the localized setting of \S 2.1, $\nu_{\cC}$ is close to its Minkowski corner value, given by \eqref{normals}. Since throughout this subsection we 
work only along the boundary $\cC$, we often denote $\nu_{\cC}$ simply by $\nu$. Thus, we write 
$$h(\nu, \cdot) = h(\nu)^{\tT} + h_{\nu \nu}\nu.$$ 
We will be careful to avoid any confusion with $\nu = \nu_S$ along $S$. 

\medskip 

   In the following, we derive equations for these remaining components 
$$u, h(\nu)^{\tT}, h_{\nu \nu}$$
along $\cC$.

  The boundary data in the target space $\cT^H$ consists of the data $([g_{\cC}], H_g, V_{\cC})$. The linearizations of these terms 
are given by 
\be \label{confb}
[g_{\cC}]'_h = \ring{h} = h^{\tT} - \frac{\tr_{\cC}h^{\tT}}{n}g_{\cC},
\ee
\be \label{mub}
(H_g)'_h = \tfrac{1}{2}\nu(\tr_g h) + \d(h(\nu)^{\tT}) - \tfrac{1}{2}h_{\nu\nu}H, 
\ee
\be \label{gauget}
(V'_h)^{\tT} = - \nabla_{\nu}h(\nu)^{\tT} + \d_{\cC}\ring{h}+ \tfrac{1}{2}d (h_{\nu \nu} + \tfrac{n-2}{n}u) - (A + H\g)h(\nu)^{\tT} - (\<D^2 x^{\a}, h\>\p_{\a})^{\tT},
\ee
\be \label{gaugen}
(V'_h)(\nu) = -\tfrac{1}{2}\nu(h_{\nu \nu}) + \tfrac{1}{2}\nu(\tr_{\cC}h^{\tT}) + \d_{\cC}(h(\nu)^{\tT}) - h_{\nu \nu}H + \<A, h^{\tT}\> - \<D^2 x^{\a}, h\>g_{\nu \a}. 
\ee 
Here the superscript $\tT$ denotes the component tangent to the boundary $\cC$. The derivations of \eqref{confb} and \eqref{mub} are elementary while 
the equations \eqref{gauget}-\eqref{gaugen} follow from the linearization of $V$ in \eqref{V'}. Namely, as noted in \cite{I}, straightforward 
computation from \eqref{V'h} gives 
$$(V'_h)^{\tT} = - \nabla_{\nu}h(\nu)^{\tT} + \b_{\cC}h{^\tT}+ \tfrac{1}{2}d h_{\nu \nu} - (A + H\g)h(\nu)^{\tT} - (\<D^2 x^{\a}, h\>\p_{\a})^{\tT},$$
Write $\b_{\cC}h^{\tT} = \d_{\cC}(\ring{h} + \frac{\tr_{\cC}h^{\tT}}{n}g_{\cC}) + \frac{1}{2}d\tr_{\cC}h^{\tT} = \d_{\cC}\ring{h} + \frac{n-2}{2n}d u$. 
Substituting this in the equation above gives \eqref{gauget}. Similar computation gives \eqref{gaugen}. Note that \eqref{gaugen} may be rewritten in 
the form 
\be \label{gaugen2}
(V'_h)(\nu) = -\tfrac{1}{2}\nu(h_{\nu \nu}) + H'_h - \tfrac{1}{2}h_{\nu \nu}H + \<A, h^{\tT}\> - \<D^2 x^{\a}, h\>g_{\nu \a}. 
\ee 

\medskip 
 
The terms $V'_h$ represent (to leading order) Neumann boundary data for $h(\nu, \cdot)$, together with given target data. Note that in these equations, 
there is no simple equation for the variation of the conformal factor 
$$u = \tr_{\cC}h^{\tT}.$$
In fact the next result shows that the linearization of the Hamiltonian constraint \eqref{Gauss} gives such an equation. 

\begin{lemma}\label{u}
The function $u$ satisfies the wave equation 
\be \label{uC}
-\tfrac{n-1}{n}\Box_\cC u-\tfrac{1}{n}R_\cC u = Y(\cT, \d^* V'_h) + E(h),
\ee
along the boundary $\cC$, where 
\be \label{uinhom}
\begin{split}
Y(\cT, \d^*V'_h)&=-R'_{\ring{h}} +\tr_g (F - \d^*V'_h) - 2 (F - \d^*V'_h)(\nu,\nu),\\
E(h)&= -2\<A'_h, A\> + 2H_{\cC}H'_h + 2\<A\circ A ,h\> - \tr_g (\d^*)'_h V + 2(\d^*)'_h V(\nu,\nu)\\
&\quad-4\Ric(\nu,\nu'_h) - \<\Ric, h\>.
\end{split}
\ee
Moreover, the initial data $u$ and $\p_t u$ for $u$ at the initial surface $\Si$ are determined by the initial and boundary data 
in $T(\cT^H)$. In particular, $\p_t u$ is determined at $\Si$ by the linearization of \eqref{velocity}. 

\end{lemma} 

\begin{proof} 
The linearization of the Gauss equation \eqref{Gauss} for the hypersurface $\cC$ gives:
\bes
2\<A'_h, A\>{ -2\<A\circ A, h\>}-2H_\cC H'_h + R'_{h^{\tT}}=R'_h-2\Ric'_h(\nu,\nu)-4\Ric(\nu,\nu'_h)
\ees
Note that 
\bes
R'_{h^{\tT}}=R'_{\tfrac{1}{n}ug_{\cC}}+R'_{\ring{h}}=-\Box_\cC u+\d_\cC\d_\cC(\tfrac{1}{n}ug_\cC)-\<\Ric_\cC,\tfrac{1}{n}ug_\cC\> + R'_{\ring{h}}=
-\tfrac{n-1}{n}\Box_\cC u-\tfrac{1}{n}R_\cC u+R'_{\ring{h}}
\ees
So we obtain:
\bes
-\tfrac{n-1}{n}\Box_\cC u-\tfrac{1}{n}R_\cC u=-R'_{\ring{h}}-2\<A'_h, A\>{ +2\<A\circ A,h\>}+2H_\cC H'_h+R'_h-2\Ric'_h(\nu,\nu)-4\Ric(\nu,\nu'_h). 
\ees
Also $R'_h=\tr \Ric'_h-\<h,\Ric\>$ and $\Ric'_h=F-\d^* V'_h-(\d^*)'_h V$. Thus
\bes\begin{split}
&-\tfrac{n-1}{n}\Box_\cC u-\tfrac{1}{n}R_\cC u\\
&=-R'_{\ring{h}} + \tr F -2F(\nu,\nu)\\
&\quad-2\<A'_h, A\> + 2H_{\cC}H'_h +2\<A\circ A,h\> - \tr [\d^*V'_h+(\d^*)'_h V] + 2[\d^*V'_h + (\d^*)'_h V](\nu,\nu) \\
&\quad - 4\Ric(\nu,\nu'_h) - \<h, \Ric\>,\\
\end{split}\ees
which proves the first statement. 

From the discussion in \S 2, cf.~Lemma \ref{corner}, note $u$ and $\p_t u$ at $\Si$ are determined by the target data 
$h_S, K'_h, \ring{h}$ in $T(\cT) \subset T(\cT^H)$. 

\end{proof}

\begin{lemma} \label{nuTlemma}
The component $h(\nu)^{\tT} = h_{\nu a} \p_a$, $a = 0, 2, \dots n$, satisfies the Neumann boundary condition 
\be \label{hnua}
\nabla_{\nu} h(\nu)^{\tT} = \hat Y(\cT) + \hat Z(u, h_{\nu \nu}) + \hat E(h)
\ee
at $\cC$, where 
\be \label{hnuaterms}
\begin{split}
&\hat Y(\cT) = \d_{\cC}\ring{h} - (V'_h)^{\tT},\\
&\hat Z(u, h_{\nu \nu}) = \tfrac{1}{2}d_{\cC}(h_{\nu\nu} + \tfrac{n-2}{n}u) \\ 
&\hat E(h) =  -\<(A + H\g)h(\nu)^{\tT}, \p_a\> - (\<D^2 x^a, h\>\p_{\a})^{\tT}.
\end{split}
\ee 
\end{lemma}

\begin{proof}
This follows easily from \eqref{gauget}. 

\end{proof}

\begin{lemma} \label{nunulemma}
The component $h_{\nu \nu}$, satisfies the Neumann boundary condition 
\be \label{hnunu}
\tfrac{1}{2}\nu (h_{\nu\nu})= \w Y(\cT) + \w E(h), 
\ee
at $\cC$, where 
\be \label{hnunuaterms}
\begin{split}
&\w Y(\cT) = H'_h - (V'_h)(\nu),\\
&\w E(h) =  -\tfrac{1}{2}H h_{\nu\nu} + \<A, h^{\tT}\> - \<D^2 x^{\a}, h\>g_{\nu \a}.
\end{split}
\ee 
\end{lemma}

\begin{proof}
This follows from \eqref{gaugen2} in the same way as above. 

\end{proof}

    For later reference, note that the $Y$ terms above depend only on the target data in $\cT^H$, with the exception of $\d^*V'_h(\nu, \nu)$ in \eqref{uinhom}. 
The $E$ terms all have coefficients of order $O(\l)$ or $O(\e)$ in the localization of \S 2.1, and all are of order zero in $h$ except for the gauge-dependent 
terms and the terms involving $A'_h$ in \eqref{uinhom}.

\begin{remark}\label{GCG}
{\rm As in the proof of Lemma \ref{uC}, by taking the linearization of the momentum constraint (Gauss-Codazzi equation) \eqref{GC} in place of the linearization of 
the Hamiltonian constraint, it is straightforward to show that the tangential vector field $h(\nu)^\intercal$ satisfies the following wave equation along $\cC$:  
\be\label{hnuC}
-\tfrac{1}{2}[\Box_\cC+\Ric_\cC]h(\nu)^\intercal=\hat Y(\cT)+\hat Z(h^\intercal) + \hat E(h),
\ee
along $\cC$, where 
\bes\begin{split}
&\hat Y(\cT, V'_h)=  { \tfrac{1}{2}dH'_h} + (F-\d^* V'_h)(\nu) , \ \ \hat Z(h^{\tT}) = \tfrac{1}{2}\nabla_\nu (\b_\cC h^{\tT}),\\
&\hat E(h)=
\b'_{h^\intercal}A - (\d^*)'_h V(\nu)+\Ric(\nu'_h)-\tfrac{1}{2}[A(\b_\cC h^{\tT})+\b'_{2A}h^{\tT}+\b_\cC\big(h_{\nu \nu}A\big)].
\end{split}
\ees
Moreover, the initial data $h(\nu)^\intercal$ and $\p_t h(\nu)^\intercal$ for $h(\nu)^\intercal$ at the initial surface $\Si$ are determined by 
initial, boundary and corner data in $T(\cT^H)$.

In fact, both the evolution equations \eqref{uC} and \eqref{hnuC} for $u$ and $h(\nu)^{\tT}$ can also be derived from the gauge equations for 
$V'_h$ at the boundary $\cC$. To see this, take the normal derivative $\nabla_{\nu}$ of \eqref{gauget} and use the facts that 
$$\nabla_{\nu}\nabla_{\nu} h(\nu)^{\tT} =  \Box_g h(\nu)^T - \Box_{\cC}h(\nu)^T + P'(\p h)$$ and 
$$-\tfrac{1}{2}\Box_g h(\nu)^{\tT} + P(\p h) = F(h(\nu)^{\tT}),$$
for $P, P'$ terms of the form in \eqref{L2}. Taking then the tangential derivative $d_{\cC}$ of \eqref{gaugen} and inserting this into the equations above 
then leads to the expression \eqref{hnuC}. Similarly, taking the divergence $\d_{\cC}$ of \eqref{gauget} and the normal derivative $\nu$ 
of \eqref{gaugen}, analogous computations give \eqref{uC}. 

  Thus imposing the Hamiltonian and momentum constraints along $\cC$ follows by imposing Dirichlet boundary data for the 
harmonic gauge. 
}
\end{remark}

  In later sections, we will repeatedly need to solve linear wave equations of the form 
\be \label{L0}
-\tfrac{1}{2}\Box_g v + p(\p v)  = \f,
\ee
for given $\f$. Here $v$ may be either scalar valued or vector-valued. Although solvability of the linear equation 
or system \eqref{L0} holds more generally, we will only need to actually solve such systems in $C^{\infty}$, so with $g$, $p$ and 
$\f$ in $C^{\infty}$. 

 The next Lemma is well-known and will be used repeatedly in \S 4. 

\begin{lemma} \label{DirSomm}
The $C^{\infty}$ system \eqref{L0} is solvable for Dirichlet boundary data, for any smooth $\f$, with standard energy estimates. 
\end{lemma}

\begin{proof} 

  The existence and uniqueness of solutions to \eqref{L0} in Sobolev spaces, with given initial data and with Dirichlet boundary data is standard, 
cf.~\cite{BS}, \cite{Sa}. We assume implicitly here that the initial data and boundary data for $v$ satisfy the requisite compatibility conditions at the 
corner $\Si$, cf.~\S 2.3. Energy estimates for such equations are also standard. Thus, the strong or boundary stable energy estimate with Dirichlet 
boundary data states that any (smooth) solution of \eqref{L0} satisfies the bound 
\be \label{DirE}
||v||_{\bar \cH^s(S_t)}^2 \leq C[||v||_{H^s(S_0)}^2 + ||\p_t v||_{H^{s-1}(S_0)}^2 + ||v||_{H^s(\cC_t)}^2 + ||\f||_{H^{s-1}(M_t)}^2],
\ee
where $C$ is a constant depending only on $g$, the coefficients of $p$ and an upper bound for $t$. Note the difference in the stronger and weaker 
norms on the left and right of \eqref{DirE}. We recall that $\cC_t = \{p \in \cC: t(p) \leq t\}$ and similarly for $M_t$. 

 \end{proof}

\begin{remark} \label{Neumann} 
{\rm One also has existence and uniqueness for smooth solutions of the equation \eqref{L0} with given Neumann boundary data $b(v) = \nu_{\cC}(v)$. 
However, in this case, there is no effective energy estimate as in \eqref{DirE}. Instead there is such an estimate but with a loss 
of derivative - or more precisely a loss of half-a-derivative, i.e.
\be \label{NeuE}
|v||_{\bar \cH^s(S_t)}^2 \leq C[||v||_{H^s(S_0)}^2 + ||\p_t v||_{H^{s-1}(S_0)}^2 + ||\nu_{\cC}(v)||_{H^{s-1/2}(\cC_t)}^2 + ||\f||_{H^{s-1}(M_t)}^2].
\ee
We refer to \cite{T} for further details. In some situations, the bound on $||\nu_{\cC}(v)||_{H^{s-1/2}(\cC_t)}$ can be improved to a stronger 
$H^s$ bound. However, it cannot be improved to an effective $H^s$ bound as in \eqref{DirE}. 
}
\end{remark}

 \section{Construction of linearized solutions.} 

In this section, we construct solutions $h$ to the linear problem $D\Phi_g^H(h) = \tau'$ by solving Dirichlet and Neumann 
initial boundary value problems for the boundary data determined by the equations \eqref{uC}, \eqref{hnua} and \eqref{hnunu} of \S 3. 

   As discussed in \S 1, let $\cD \subset T(\cT^H)$ be the subset of target data $\tau'$ for which the $C^m$ norms of $\tau'$ are uniformly bounded 
in $m$: thus $\cD = \cup_{K\in \bZ^+} \cD_K$, where $\cD_K$ is the space of $\tau'$ such that 
\be \label{Mbound}
||\tau'||_{C^m} \leq K,
\ee
for all $m \in \bZ^+ \cup \{0\}$.\footnote{All the results here hold for derivative growth as in Footnote 2 in place of the uniform bound \eqref{Mbound}.}

 Observe that $\cD$ is dense in $T(\cT^H)$ in the $C^{\infty}$ semi-norm topology on $\cT^H$. To see this, by decomposing the relevant tensors into their 
components, it suffices to show that the space of $C^{\infty}$ functions on a compact domain $D \subset \bR^{n+1}$ for which all $C^{m}$ norms 
are uniformly bounded by a constant as in \eqref{Mbound}, is a dense subspace of $C^{\infty}(D)$ with the semi-norm topology. However, this is a standard 
result, which follows easily from the Stone-Weierstrass theorem; all $\tau'$ with polynomial component functions are in $\cD$. 
The same notation and discussion holds for the background metric $g$. 
 
\subsection{Local linearized solutions.}

   Throughout this section, the local rescaled metric $\w g$ as in \S 2.1 will be denoted by $g$. Also, as in \S 3, the unit normal $\nu_{\cC}$ to $\cC$ is generally 
denoted simply by $\nu$ and is extended into bulk region $\w U$ as in \eqref{nuC}. Clearly, $\nu = \nu_{\cC}$ is close to its Minkowski corner value, 
given by \eqref{normals}. 

     The main result of this section is the following dense-range result. Let 
\be \label{tau'1}
\tau' = (F, (\g_S', \k', \nu_S', V_S'), (\s', \ell', V_{\cC}'), \a'), 
\ee
be a general element in $T_{\tau}(\cT^H)$, $\tau = \Phi^H(g)$. Define a Sobolev $H^s$ norm on the target space data in $T_{\tau}(\cT^H)$ as follows: 
for $\tau'$ as in \eqref{tau'1},  
\be \label{tnorm}
\begin{split}
||\tau'||_{H^s(\cT^H)} = [ ||F||_{H^s(\w U)} & + [ ||\s'||_{H^{s+1}(\cC)} + ||\ell'||_{H^{s-1}(\cC)} + ||V'_{\cC}||_{H^s(\cC)}] + ||\a'||_{H^s(\Si)} \\
& + [ ||\g'_S||_{H^{s+1}(S_0)} + ||\k'||_{H^{s}(S_0)} + ||\nu'||_{H^s(S_0)} + ||V'_S||_{H^s(S_0)} ].
\end{split}
\ee
 On the right in \eqref{tnorm}, the top line consists of bulk, boundary and angle terms, while the bottom line consists of initial data terms. Note the shift in the 
derivative index $s$ on the target compared with the domain $TMet(M)$. Of course $S_0$ and $\cC$ are the corresponding 
local domains in $\w U$. As above, the local rescaled quantities $\w g$, $\w h$ are denoted by $g$, $h$.

\begin{theorem}\label{locexist}
There exists $\e > 0$ such that, for any smooth metric $g \in \cD$ on $\w U$ which is $C^{\infty}$ $\e$-close to a standard Minkowski corner metric $g_{\a_0}$ 
as in \eqref{eps}, the linearization $D\Phi_g^H$ at $g$ satisfies: given any $C^{\infty}$ target data 
$\tau' \in \cD$ on $\w U$, there exists a $C^{\infty}$ variation $h \in T_g Met(\w U)$, such that 
\be \label{DwPhisol}
D\Phi_g^H(h) = \tau'.
\ee
Thus, the equation 
\be \label{LhF}
L(h) = F,
\ee
has a $C^{\infty}$ solution $h$ such that 
\be \label{IC1}
(g_S, K)'_h = (\g_S', \k'), \ \ (\nu_S)'_h = \nu_S', \ \ V'_h = V_S' \ \mbox{ on }S,
\ee
and
\be \label{BC1}
(([g_\cC]), H)'_h = (\s', \ell'), \ \ V'_h = V_\cC' \ \mbox{ on }\cC.
\ee
Further the constructed solution $h$ depends smoothly on the data $\tau'$. 
\end{theorem}
  
 \begin{proof} 

  The proof proceeds by a standard Picard-type iteration process. As in \S 3, we work with separate equations for each of the terms 
\be \label{comp}
h^{\tT}, \ h(\nu)^{\tT}, \ h_{\nu \nu},
\ee
comprising $h$. The equations for $u$, $h(\nu, \cdot)$ along the boundary $\cC$ are those described in \S 3.

\medskip

0. In this step, we solve all the relevant equations in \S 3 without all the coupling terms, i.e.~the $P$-terms in the bulk and the $E$-terms on $\cC$ are 
set to zero, and we use only the given target data in $T(\cT^H)$ (i.e.~$F$ and $Y$ terms) for bulk, initial and boundary data. This is done step-by-step 
for each of the terms in \eqref{comp}, 

  We begin with the construction of the gauge field. Thus let $V_0'$ be the solution to the modified gauge equation 
\be \label{modV'}
-\tfrac{1}{2}(\Box_g V'_0+ \Ric_g (V'_0)) + \tfrac{1}{2}\nabla_{V} V'_0 = \b F,
\ee
(so the `error' term $\b'_h \Ric + O_{2,1}(V,h)$ is dropped from \eqref{V'}) and with initial and boundary data given by the target data $V'_S$ and $V'_{\cC}$.  
By the Dirichlet energy estimate \eqref{DirE}, we have 
$$||V_0'||_{\bar \cH^s(S_t)} \leq C||\tau'||_{H^s(\cT^H)}.$$

  Next solve the modified evolution equation \eqref{uC} for $u_0$ along $\cC$, 
\be \label{Hlin2} 
-\tfrac{n-1}{n} \Box_\cC u _0 -\tfrac{1}{n}R_{\cC}u_0 = R'_{\s'} - \tr(F-\d^* V_0') + 2(F-\d^* V_0')(\nu,\nu),
\ee
so that $E$ is set to zero and $V'_h$ is set to $V_0'$. Recall from Lemma \ref{u} that the initial conditions $u_0 |_{\Si}$ and $\p_t u_0 |_{\Si}$ are determined 
by the target data in $T(\cT)$. Solving \eqref{Hlin2} on $\cC$ with this data gives the ($C^{\infty}$) Dirichlet boundary value for $u_0$ along 
$\cC$. Standard energy estimates then give 
$$||u_0||_{H^s(\cC)} \leq C||\tau'||_{H^s(\cT^H)}.$$
Since the Dirichlet data $\s'$ is given in $\tau'$, the Dirichlet boundary data for $h_0^{\tT} = \s' + \frac{u_0}{n}g_{\cC}$ is thus also given. 
With such Dirichlet boundary data, we then solve the bulk equation 
\be \label{hT0}
-\tfrac{1}{2}\Box_g h_0^{\tT}   = F^{\tT}, 
\ee
with the given initial conditions from $\tau'$ on $\w U$. By \eqref{DirE}, the solution $h_0^{\tT}$ satisfies the energy estimate 
\be \label{H0Test}
||h_0^{\tT}||_{\bar \cH^s(S_t)} \leq C||\tau'||_{H^s(\cT^H)}.
\ee

  Next define $(h_0)_{\nu\nu}$ to be the unique solution to 
$$-\tfrac{1}{2}\Box_g (h_0)_{\nu\nu} = F_{\nu \nu},$$
with the Neumann data $\nu((h_0)_{\nu \nu})$ given by \eqref{hnunu} with $\w Y(\cT) = \ell' - V_0'(\nu)$ and $\w E = 0$. 
From the Neumann estimate \eqref{NeuE} we obtain 
\be \label{0nunu}
||(h_0)_{\nu\nu}||_{\bar \cH^s(S_t)} \leq C||\tau'||_{H^{s+1}(\cT^H)}.
\ee

   Turning next to $h(\nu)^{\tT}$, define $h_0(\nu)^{\tT}$ in $\w U$ to be the unique solution to 
$$-\tfrac{1}{2}\Box_g h_0(\nu)^{\tT} = F(\nu)^{\tT},$$
with the Neumann data $\nabla_{\nu}h_0(\nu)^{\tT}$ given by \eqref{hnua}. Here we set $\hat Y(\cT)$ from 
the target data in $T(\cT)$ and $\hat Z$ from the data above, so in particular $V'_h = V_0'$, $u = u_0$ and $h_{\nu\nu} = (h_0)_{\nu\nu}$; in 
addition, set $\hat E = 0$. Again, from the Neumann estimate \eqref{NeuE} we obtain 
\be \label{0nua}
||h_0(\nu)^{\tT}||_{\bar \cH^s(S_t)} \leq C||\tau'||_{H^{s+2}(\cT^H)}.
\ee
Note the significant loss-of-derivative here and in \eqref{0nunu} compared with the previous estimates. 

 Summing up the solutions above gives a solution $h_0$ to 
\be \label{h0F}
-\tfrac{1}{2}\Box_g h_0 = F,
\ee
in $\w U$ with the prescribed initial and boundary conditions from $\tau'$. Further 
$$||h_0||_{\bar \cH^s(S_t)} \leq C||\tau'||_{H^{s+2}(\cT^H)}.$$

 The next stages are error, so $O(\e)$ adjustments to $h_0$, and so are denoted $E_1$, $E_2$, etc. In contrast to 
Step 0, we now set all the target data in the equations from \S 3 to zero while the $E$ terms in these equations become non-zero. 

1.  First determine the correction to the gauge by letting $V_1'$ be the solution to the equation
$$-\tfrac{1}{2}(\Box_g V'_1+ \Ric_g (V'_1)) + \tfrac{1}{2}\nabla_{V} V'_1 = \b'_{h_0} \Ric + O_{2,1}(V,h_0),$$
with zero initial and Dirichlet boundary data. As above, we then have the estimate 
$$||V_1'||_{\bar \cH^s(S_t)} \leq C\e ||h_0||_{H^s(\w U)} \leq C\e || \tau'||_{H^{s+2}(\cT^H)}.$$

  Next, solve the evolution equation \eqref{uC} for $u_1$ along $\cC$ with $Y(\cT, V'_h) = Y(0, V_1')$, $E = E(h_0)$ and zero initial data. This gives 
Dirichlet boundary data for $u_1$ on $\cC$ and since the coefficients of $E$ are $O(\e)$ in $C^{\infty}$ by \eqref{eps}, 
$$||u_1||_{H^s(\cC)} \leq C\e ||h_0||_{H^s(\cC)} \leq C\e ||\tau'||_{H^{s+2}(\cT^H)}.$$
We set $\ring{E_1} = 0$ on $\cC$ and so $E_1^{\tT} = u_1 g_{\cC}$. Using this Dirichlet boundary value, we then solve 
$$-\tfrac{1}{2}\Box_g E_1^{\tT} = -P^{\tT}(h_0),$$
in $\w U$ with zero initial data. The Dirichlet energy estimate then gives 
$$||E_1^{\tT}||_{\bar \cH^s(S_t)} \leq C \e ||\tau'||_{H^{s+2}(\cT^H)}.$$

Next let $(E_1)_{\nu \nu}$ on $\cC$ be the unique solution to the Neumann IBVP 
$$-\tfrac{1}{2}\Box_g (E_1)_{\nu \nu} = -P_{\nu \nu}(h_0)$$
in $\w U$, with zero initial data and Neumann boundary data as in \eqref{hnunu} with $\w Y(\cT) = 0$ and $\w E =  -\tfrac{1}{2}H (h_0)_{\nu\nu} + 
\<A, h_0^{\tT}\> - \<D^2 x^{\a}, h_0\>g_{\nu \a}$. Then  
$$||(E_1)_{\nu\nu}||_{\bar \cH^s(S_t)} \leq C \e ||\tau'||_{H^{s+3}(\cT^H)}.$$

  Next, for $E_1(\nu)^{\tT}$, solve the IBVP 
$$-\tfrac{1}{2}\Box_g E_1(\nu)^{\tT} = -P(\nu)^{\tT}(h_0),$$
on $\w U$ with zero initial data and with Neumann boundary data in \eqref{hnua} with $\hat Y(\cT, V'_h) = (0, V_1')$, $\hat Z = Z(u_1, (E_1)_{\nu \nu})$ 
and $\hat E = \hat E(h_0)$. As previously, we then have the bound  
$$||E_1(\nu)^{\tT}||_{\bar \cH^s(S_t)} \leq  C\e||\tau'||_{H^{s+4}(\cT^H)}.$$

  As in Step 0, summing these components together then gives $E_1$ solving 
$$-\tfrac{1}{2}\Box_g E_1  = -P(h_0),$$
in $\w U$. The term $E_1$ has zero target initial data and zero target boundary data and satisfies the estimate 
$$||E_1||_{\bar \cH^s(S_t)} \leq C\e ||\tau'||_{H^{s+4}(\cT^H)}.$$

2. Next, $E_2$ is constructed in the same way as $E_1$ above, with $E_1$ in place of $h_0$ and $E_2$ in place of 
$E_1$. Thus, first $V_2'$ solves 
$$-\tfrac{1}{2}(\Box_g V'_2+ \Ric_g (V'_2)) + \tfrac{1}{2}\nabla_{V} V'_2 = \b'_{E_1} \Ric + O_{2,1}(V,E_1),$$
and with zero initial and Dirichlet boundary data. As previously, we then have the estimate 
$$||V_2'||_{\bar \cH^s(S_t)} \leq C\e^2  ||\w \tau'||_{H^{s+4}(\cT^H)}.$$
Similarly, $E_2$ solves 
$$-\tfrac{1}{2}\Box_g E_2  = -P(E_1),$$
in $\w U$, with zero target  initial data and zero target boundary data. The boundary data for $E_2$ is constructed by 
setting $Y(\cT, V'_h) = (0, V'_2)$, $E = E(E_1)$ in \eqref{uinhom}, $\hat Y = 0, \hat E = \hat E(E_1)$ in \eqref{hnunu}, 
and $\w Y = 0, \w Z = \w Z(u_2, (E_2)_{\nu\nu}), \w E = \w E(E_1)$ in \eqref{hnua}. As before, we then obtain the bound 
$$||E_2||_{\bar \cH^s(S_t)} \leq C_2\e ||E_1||_{\bar \cH^{s+2}(S_t)} \leq C_2 C_1\e^2||\tau'||_{H^{s+6}(\cT^H)},$$
where the constants $C_i$ reflect the dependence of the estimates on $s$. 

   Continuing inductively in this way, it follows that if $\tau' \in \cB$, i.e.~$\tau' \in \cB_K$ for some $K$, then the sequence $\{E_m\}$, 
$m \geq 1$, satisfies the bound 
\be \label{Embound}
||E_m||_{\bar \cH^s(S_t)} \leq C_m\e ||E_{m-1}||_{\bar \cH^{s+2}(S_t)} \leq C_m C_{m-1} \e^2 ||E_{m-2}||_{\bar \cH^{s+4}(S_t)} \leq \cdots \leq 
(\prod_1^m C_i) \e^m ||\tau'||_{H^{s+2m+2}(\cT^H)}.
\ee
The constants $C_i$ derive from the energy estimates \eqref{DirE} and \eqref{NeuE}. These depend on the $C^j$ norm, $0 \leq j \leq s+2m$, of the 
background metric $g$, due to the commutator $[\p_{x^{\a}}, \Box_g] = O(\p^2 g)$ and its iterates. However, since by assumption 
$g \in \cD$, there is a fixed constant $C$, independent of $s$, such that $\prod_1^m C_i \leq C^m$, so that  
$$||E_m||_{\bar \cH^s(S_t)} \leq C^m \e^m K.$$
By choosing $\e$ sufficiently small, it follows that the sequence $h^k = h_0 + \sum_1^k E_i$ converges to a limit $h \in C^{\infty}$. Similarly, 
the gauge approximations $V^k = \sum_{j=0}^k V_j$ converge in $C^{\infty}$ to a limit gauge field $V$. 
By construction we have 
$$-\tfrac{1}{2}\Box_g(h_0 + \sum E_i)  = F - P(h_0 + \sum E_i),$$
so that 
\be \label{Lbarh}
L(h) = F.
\ee
Hence $h$ solves \eqref{LhF}. By construction, it is easy to see that $h$ satisfies \eqref{IC1} as well as $[g_{\cC}]'_{h} = \s'$ in \eqref{BC1}. 
We claim next that  
$$V' = V'_h.$$ 
To see this, note that by construction, both $V'$ and $V'_h$ have the same initial values on $S$ and $(V')^{\tT} = (V'_h)^{\tT}$ on $\cC$. Regarding the 
normal component, first observe that by construction, the limit $u = \lim \sum u^k$ satisfies \eqref{uC}, i.e. 
\be \label{Hlin2}
\begin{split} 
-\tfrac{n-1}{n} \Box_{\cC} u - \tfrac{1}{n}R_{\cC}u & =  R'_{\ring{h}} - \tr(F - \d^* V') + 2(F - \d^* V')(\nu,\nu)\\
 &+2\<A'_h, A\> - 2H_\cC H'_h - 2\<A\circ A, h\> + 4\Ric(\nu,\nu'_h) + \<\Ric, h\>\\
&+ \tr [(\d^*)'_h V] - 2[(\d^*)'_h V](\nu,\nu), \\
\end{split}
\ee
along $\cC$. Moreover, by \eqref{Lbarh}, $(\Ric + \d^*V)'_h = F$ and so applying the linearization of the Hamiltonian constraint as before shows 
that \eqref{Hlin2} also holds with $V'_h$ in place of $V'$. Hence  
$$\tr(\d^*V')-2\d^* V'(\nu,\nu) = \tr(\d^* V'_h) - 2\d^*  V'_h(\nu,\nu) \mbox{ on }\cC.$$
Since $(V')^{\tT} = (V'_h)^{\tT}$, this equation shows that 
$$\nu \<V', \nu\> = \nu\<V'_h, \nu\> \ \ {\rm at} \ \ \cC.$$
Since both $V'$ and $V'_h$ satisfy the same gauge equation \eqref{V'}, it then follows from uniqueness of solutions to the mixed Dirichelt-Neumann 
boundary value problem that $V' = V'_h$. From this and from \eqref{hnunu}, it follows that along $\cC$, 
$$H'_h = \ell'.$$ 
Finally, the angle variation satisfies $\a'_h = \a'$, since the angle variation is a term in $\p_t u$ on $\Si$ which is part of the construction of $h$. 

Also, the construction above  is clearly smooth in the data $T(\cT^H)$ and in $g$. This completes the proof of Theorem \ref{locexist}.  

\end{proof}

\begin{remark}\label{derivloss}

  The iteration process used in the construction of the solution $h$ loses derivatives at each stage, due to the use of the Neumann boundary data in 
Lemmas \ref{nuTlemma} - \ref{nunulemma}. Currently, it is unknown if Theorem \ref{locexist} can be improved to a full surjectivity of the map $D\Phi_g^H$ 
at, say,  generic $g \in C^{\infty}$; this typically requires the use of stronger apriori energy estimates (as used in \cite{I}, \cite{II}) which are not known to hold in the 
current setting.   
  
 \end{remark}

\subsection{Global linearized solutions} 

\begin{theorem}\label{globexist}
Let $(M, g)$, $g \in \cD$, be a smooth globally hyperbolic Lorentzian manifold, $M \simeq I \times S$, with compact Cauchy surface $S$ having compact 
boundary $\p S = \Si$ and with timelike boundary $(\cC, g_{\cC})$. Then for any smooth $\tau' \in \cD \subset T_{\tau}(\cT^H)$, the equation 
\be \label{wL3}
D\Phi_g^H(h) = \tau'
\ee
has a $C^{\infty}$ solution $h$. The constructed solution $h$ depends smoothly on the data $\tau'$. 
\end{theorem} 
 
 \begin{proof}

  Given the previous local results, the proof of Theorem \ref{globexist} is essentially standard. We provide the details below for completeness. 
    
  Given $(M, g)$ as in \S 1, first choose an open cover $\{U_i\}_{i=1}^N$ of the corner $\Si$, where each $U_i$ is small enough so that the local 
existence result Theorem \ref{locexist} holds for each $U_i$. Additionally, choose an open set $U_0$ in the interior $M\setminus \cC$ such that 
$(\cup_{i=1}^N U_i)\cup U_0$ also covers a tubular neighborhood of the initial surface, i.e. $S\times [0,t']$ for some time $t'>0$.

Let $\{\w U_i\}_{i=0}^N$ be a thickening of the cover $\{U_i\}_{i=0}^N$, so that $U_i \subset \w U_i$, as described in \S 2.1. Let 
$\rho_i$ be a partition of unity subordinate to the cover $\{\w U_i\}_{i=0}^N$, so that ${\rm supp}\rho_i\subset \w U_i$, $\rho_i = 1$ on 
$U_i$ and  $\sum_i \rho_i=1$. Let 
$$\tau' =\big(F, (\g', \k', \nu', V_S'), (\s', \ell',V_{\cC}'), a' \big),$$
be arbitrary $C^{\infty}$ data in $T(\cT^H)$ on $M$. Then the data $\rho_i \tau'$ has compact support in the sense of \S 2.1 in $\w U_i$ 
and by Theorem \ref{locexist} there exists a solution $h_i$ in $\w U_i$ satisfying 
$$D\Phi^H(h_i) = \rho_i \tau'.$$ 
Moreover, as noted at the end of \S 2.1, by the finite propagation speed property, the solution $h_i$ has compact support in the sense of \S 2.1 for 
$t \in [0,t_i]$ for some $t_i > 0$. Thus, $h_i$ extends smoothly as the zero solution on $M_{t_i} \setminus \w U_i$. 

   Similarly, in the interior region $U_0 \subset \w U_0$, let $h_0$ be the solution to the Cauchy problem 
$$L(h_0) = \rho_0 F \ \mbox{ in } \w U_0, \ \ (g_S, K_S, \nu_S, V_g)'_h = \rho_0(\g', \k', \nu', V_S') \ \mbox{ on }S_0.$$
Then again $h_0$ has compact support in $\w U_0$ for some time interval $[0, t_0]$ and as above extends to $M_{t_0}$. We relabel so that 
$t_0$ is a common time interval for all $h_i$, $i \geq 0$. 

  The sum 
$$h = \sum_{i=0}^N h_i,$$
is thus well-defined on $M_{t_0}$ and by linearity
$$D\Phi^H(h) = \tau'.$$ 
By the proof of Theorem \ref{locexist}, the constructed $h$ satisfies the estimate
\be \label{enest2}
||h||_{C^{m}(S_t)} \leq K,
\ee
for all $m$ and for $t \leq t_0$. Here $K$ is a constant $K = K(t_0, \tau')$, which depends only on the background metric $g$ and the 
target data $\tau' \in \cD$. 

  One may now continue the solution $h$ past the time $t_0$ to all $t < \infty$ in the usual way, assuming of course that the background metric 
$g$ satisfies this condition. Briefly, let $I_{t_0}$ denote the target initial data of $h$ at the Cauchy slice $\{t=t_0\}$ and form the target data 
$\tau'$ by replacing the initial data in $\tau'$ with $I_{t_0}$. Then applying Theorem \ref{locexist} with initial slice $S_{t_0}$ and using \eqref{enest2} 
gives the continuation of $h$ to $[0, t_1]$, with $t_1 > t_0 + \mu$, where $\mu$ depends only on the boundary data $b'_h$ on $\cC$. 
Assuming then $g$ is defined for all time $t$, since then $b'_h$ is also globally defined on $\cC$ for all $t$, the same holds for $h$. 
  
\end{proof}

\section{Uniqueness}

   In this section, we prove the uniqueness result stated in Theorem \ref{ThmI}. This holds for general smooth $\tau' \in \cT^H$, not necessarily 
$\tau' \in \cD$; it may be considered as a Holmgren-type uniqueness result in this setting.  Note that the usual proof of uniqueness of (linearized) solutions 
of an IBVP is via energy estimates. However, as noted above, such energy estimates are not available here and so a different method is needed. 
The basic idea is to use the self-adjointness property derived from an action or variational principle; this has previously been noted and used 
in the elliptic (Riemannian) case, cf.~\cite{A1}, \cite{AK}.

  For convenience, the uniqueness result is phrased in the global setting where $S$ is a compact Cauchy slice with boundary corner 
$\Si$ a compact (codimension two) surface; the same result holds locally, for compactly supported data as in \S 2.1, cf.~Remark \ref{localunique},

   Let $C^{\infty}_0$ be the space of $C^{\infty}$ deformations $h$ with zero initial, boundary, gauge and corner data, i.e. 
\be \label{zero}
\big((h_S,K'_h, (\nu_S)'_h, V'_h), (\ring{h}, H'_h, V'_h), \a'_h \big) = 0.
\ee
We set $S = S_0 = \{t = 0\}$. 

\begin{theorem}\label{unique}
Let $g$ be a vacuum Einstein metric, $\Ric_g = 0$ on $M$ as above. Suppose the linearized operator $L$ has dense range on $C^{\infty}_0$, i.e. 
for any $F$ within a dense subset $C^{\infty}$, the equation 
$$L(h) = F,$$
has a solution $h \in C_0^{\infty}$. {\rm (}Compare with Theorem \ref{globexist} {\rm )}. Then $L$ has trivial kernel on $C^{\infty}_0$, 
\be \label{KerL}
Ker L = 0.
\ee
\end{theorem}

  The proof will be carried out in several steps.  We first pass to a different gauge, the divergence-free gauge, associated with the Einstein tensor 
$$E_g = \Ric_g - \frac{R_g}{2}g.$$ 
Note that the Bianchi identity gives $\d E = 0$, equivalent to the previously used $\b \Ric = 0$. The reason for passing from $\Ric$ to $E$ is that 
$E$ has a natural variational or action principle; namely it arises as the Euler-Lagrange operator for the variation of the Einstein-Hilbert action. 
Such variational operators have natural (formal) self-adjoint properties, after taking into account suitable boundary terms. This structure is not 
available for the bare Ricci tensor $\Ric$. 

   To begin, choose any fixed background vacuum metric $\w g$ near $g$ and consider the bulk gauged operator 
\be \label{newPhi}
\w \Phi^H(g) = E_g + \d^* \d_{\w g} g.
\ee
The linearization of $\w \Phi^H$ at the background vacuum metric $g = \w g$ in the bulk is given by  
\be \label{newlin}
\w L(h) = D \w \Phi^H(h) = \tfrac{1}{2}D^*Dh - Rm(h)  - \tfrac{1}{2}D^2 \tr \, h + \tfrac{1}{2}(\Box \tr \, h- \d \d h)  g,
\ee
compare with \eqref{L}. Note that the new $\w \Phi^H$ in \eqref{newPhi} depends on a choice of background metric $\w g$. This is not case with 
the original operator $\Phi^H$ in \eqref{PhiH}, which instead depends on a (fixed) choice of coordinates $x^{\a}$. The formula \eqref{newlin} changes by 
the addition of lower order terms when $g \neq \w g$, so we really have a family of operators $\w L_{\w g}$. 

  Let $T(\w \cT^H)$ be the linearized target data space $T(\cT^H)$, with the gauge $V'_h$ on $S$ and $\cC$ replaced  by the new gauge data $\d h$ on $S$ 
and $\cC$. Similarly, let $\w C^{\infty}_0$ be the associated space where this target data vanishes, so that \eqref{zero} holds with $\d h = 0$ 
in place of $V'_h = 0$. Although not needed for the proof, we note that Lemma \ref{Gauge-lemma2} holds as before: if $\Ric_g = 0$ and if 
$\w L(h)=F$ satisfies $\d_{\w g}F=0$ in $M$ with $F(\nu_S, \cdot) = 0$ on $S$, then $h\in \w C_0^\infty$ satisfies $\p_t \d h = 0$ on $S$, 
and hence $\d h = 0$ on $M$. The proof of this is the same as before. 

  Now the full boundary $\dm$ of the compact domain $M$ consists of the timelike cylinder $\cC$, the bottom and top Cauchy surfaces $S = S_0$, 
$S_1 = \{t = 1\}$ (the value $t = 1$ is for convenience) and the two corners $\Si = \Si_0$, $\Si_1$. We consider pairs of deformations $h$, $k$ with 
\be \label{01}
h \in \w C^{\infty}_0 \ \ {\rm and} \ \  k \in \w C^{\infty}_1,
\ee
where $\w C^{\infty}_1$ consists of all deformations $k$ for which the initial, gauge, corner (and boundary) data vanish on the top slice $S_1$ 
(in place of $S_0$ for $h$). 

\begin{proposition}\label{selfad}
For $h \in \w C^{\infty}_0$ and $k \in \w C^{\infty}_1$, we have the formal self-adjoint property
\be \label{selfadj}
\int_M \<\w L(h), k\> dv_g = \int_M \<h, \w L(k)\> dv_g,
\ee
where the pairing is with respect to the Lorentz metric $g$. 
\end{proposition}

\begin{proof} 

Consider the Lagrangian in 4-dimensions given by, cf.~\cite{A2}, \cite{AA2}, 
\be \label{EH}
S(g) = \int_M R_g dv_g + \tfrac{2}{n}\int_{\cC}H_{\cC} dv_{\g}.
\ee
This is a slight modification of the Einstein-Hilbert action with Gibbons-Hawking-York boundary term (in units where $16\pi G = 1$). The first variation 
is given by 
\bes\begin{split}
D_g S(h) =& -\int_M \<E_g, h\>  + \int_{\cC}[\<\ring{\pi}, \ring{h}\> + \tfrac{4}{n}H'_h] + \int_{S_1}\theta(h) - \int_{S_0}\theta(h)\\
 &{ +\int_{\Si_1}h(\nu,T)-\int_{\Si_0}h(\nu,T)}
 \end{split}\ees
Here $\pi = A_{\cC} - H_{\cC}g_{\cC}$ is the conjugate momentum and $\ring{\pi}$ is its trace-free part. The term $\theta$ is the (pre)-symplectic potential 
given by $\theta(h) = -\star(d tr_g h - div_g h)$; see e.g.~\cite{HW}, \cite{AA2}. Here and below we drop 
the notation for the various volume forms. 

The $2^{\rm nd}$ variation is then given by 
\be \label{2nd}
\begin{split}
 D^2 S_g(h, k) = & -\int_M [\<(E_g)'_k, h\> - 2\<E_g \circ h, k\> + \tfrac{1}{2}\<E_g, h\>\tr k ] \\
 &+ \int_{\cC}[ \<(\ring{\pi})'_k, \ring{h}\> - 2\<\ring{\pi} \circ k, \ring{h}\> + \tfrac{4}{n}(H'_h)'_k + \tfrac{1}{2}[\<\ring{\pi}, \ring{h}\> + \tfrac{4}{n}H'_h]\tr_{\cC}k  \\
&+ \int_{S_1}(\theta_h)'_k + \tfrac{1}{2}\theta(h)\tr_{S_1}k -  \int_{S_0}(\theta_h)'_k + \tfrac{1}{2}\theta(h)\tr_{S_0}k\\
 &{ +(\int_{\Si_1}-\int_{\Si_0})h(\nu'_k,T) + h(\nu, T'_k) + \tfrac{1}{2}h(\nu,T)\tr_\Si k}.
\end{split}
\ee

    Throughout the following discussion, we work with $h \in \w C^{\infty}_0$, $k \in \w C^{\infty}_1$ as in \eqref{01}. The corner integrals vanish 
since $h,k$ have zero initial data. The boundary term over $\cC$ above vanishes on $\w C^{\infty}_0$ and $\w C^{\infty}_1$ - except for the term 
$(H'_h)'_k$. Note however that this term is symmetric in $h$ and $k$.  

  At the top and bottom slices, the difference $D^2 S_g(h, k) - D^2 S(k, h)$ of the two terms is the (pre)-symplectic form $\O(h, k)$ evaluated at the 
top and bottom Cauchy slices, cf.~\cite{HW}. Simple computation, cf.~\cite{AA2} for example, shows this is of the form 
$$\O_S(h,k) = -\int_S \<K'_h, k\> - \<K'_k, h\> + \tfrac{1}{2}[\tr_S h(\<K_S, k\> + 2(\tr K_S)'_k) - \tr_S k(\<K_S, h\> + 2(\tr K_S)'_h)]. $$
This expression involves only the initial data of $h$ and $k$ at $S$. Since the full initial data of 
$h$ vanishes at $S_0$, (since $h \in \w C^{\infty}_0$), 
$$\O_{S_0}(h, k) = 0.$$
Similarly since the full initial data of $k$ vanishes at $S_1$, (since $k \in \w C^{\infty}_1$)
$$\O_{S_1}(h, k) = 0.$$

  Thus on-shell where $E_g = 0$, by the symmetry of the $2^{\rm nd}$ derivative operator, it follows from \eqref{2nd} that for all 
  $h \in \w C^{\infty}_0$, $k \in \w C^{\infty}_1$, 
$$\int_M \<(E_g)'_k, h\> = \int_M \<(E_g)'_h, k\>.$$
Note this conclusion does not involve any choice of gauge. Since $\w L(h) = E'_h + \d^* \d h$ and since for $h \in \w C^{\infty}_0$, 
and $k \in \w C^{\infty}_1$, 
$$\int_M \<\d^* \d k, h\> = \int_M \<\d^* \d h, k\>,$$
it follows that 
$$\int_M \<\w L(k), h\> = \int_M \<\w L(h), k\>.$$

\end{proof}

  At this point, we need the following. 
\begin{lemma}\label{gshift}
For $\Ric_g = E_g = 0$, suppose the operator $L = L_g: C^{\infty}_0 \to T(\cT^H)$ has dense range. Then the operator $\w L: \w C^{\infty}_0 
\to T(\cT^H)$ (with $\w g = g$) also has dense range, i.e.~for any $G \in \cD$, there exists $\w h \in \w C^{\infty}_0$ such that 
$$\w L(\w h) = G.$$
Consequently, 
\be \label{KwL}
Ker \w L = 0,
\ee
on $\w C^{\infty}_0$. 
\end{lemma}

\begin{proof} 
  The first statement is just a shift in gauge, and so is not particularly surprising. To give the details, observe that 
any $G\in C^\infty$ can be decomposed as 
\bes
G=G_0+\d^* Y
\ees
where $\d G_0=0$ and $Y$ is the vector field such that $\d\d^* Y=\d G$ in $M$ with vanishing initial and boundary data. Then to solve 
$\w L(\w h)=G$, it suffices to solve for $\w h\in\w C^\infty_0$ such that 
\be\label{E''}
E'_{\w h}=G_0, \ \d \w h=Y
\ee
By hypothesis, the equation $L(h) = F$ is solvable for any $F \in \cD \subset C^{\infty}$, with $h \in C^{\infty}_0$. In particular, choose $F=G_0-\frac{1}{2}\tr G_0 g$. 
It follows from $\d G_0=0$ that $\b F=0$ and hence by Lemma \ref{Gauge-lemma2}, $V'_h = 0$ in $M$. In turn, this implies that $Ric'_h = F$ in 
$M$. Since we are working on-shell where $\Ric_g = 0$, 
$$E'_h = \Ric'_h - \tfrac{1}{2}\tr \Ric'_h \, g = F - \tfrac{1}{2}\tr F \, g=G_0.$$ 

Next we transform or shift $h$ by a diffeomorphism $\d^*X$, i.e.~set $\w h = h + \d^*X$, where $X = 0$ on $\cC$, $X = \p_t X = 0$ on 
$S$, so that $\d \w h = Y$. Then $\w h$ solves \eqref{E''}, since $E'_{\d^*X} = 0$ on-shell. Since $\w h$ is now in $\w C^{\infty}_0$, this 
completes the proof of the first statement. 

  To prove \eqref{KwL}, suppose $\w L(\w h) = 0$, with $\w h \in \w C^{\infty}_0$. Then by Proposition \ref{selfad}, 
$$\int_M \<\w L(k), \w h\> = 0,$$
for all $k \in \w C^{\infty}_1$. By the above, the equation $\w L(k) = G$ is solvable for any $G$ in the dense subset $\cD \subset C^{\infty}$ with 
$k \in \w C^{\infty}_1$. Since the Lorentzian pairing in \eqref{selfad} is non-degenerate and continuous, it follows that $\w h = 0$, so that $Ker \w L = 0$. 

\end{proof}

  We now complete the proof of Theorem \ref{unique}. 
  
\begin{proof}

  Naturally, the proof is essentially the reverse of Lemma \ref{gshift}. Thus, to prove \eqref{KerL}, suppose $L(h) = 0$ with $h \in C^{\infty}_0$. Without 
loss of generality, we may assume $V_g = 0$ and so by Lemma \ref{Gauge-lemma2}, $V'_h = 0$ and $\Ric'_h = 0$. Choose the vector field $X$ with 
$X = \p_t X = 0$ on $S$, $X = 0$ on $\cC$ such that $\w h = h - \d^*X$ satisfies $\w h \in \w C^{\infty}_0$, i.e.~$\d \w h = 0$. Then Lemma \ref{gshift} 
above implies that $\w h = 0$, so $h = \d^*X$. It follows that 
$$V'_{\d^*X} = 0$$
on $M$. Thus, 
\be \label{Xg}
\b \d^*X - \<D^2 x^{\a}, \d^*X\>\p_{x^{\a}} = 0 \ \ {\rm on} \ \ M.
\ee
The second term is first order in $X$ so \eqref{Xg} is a hyperbolic system of wave equations (as in Lemma \ref{Gauge-lemma2}). Since $X = 0$ to 
first order on $S$ and $X = 0$ on $\cC$, it follows that $X = 0$ on $M$. Hence $h = 0$ which proves the result. This completes the proof of 
Theorem \ref{unique}.  

\end{proof}

\begin{remark}\label{localunique}
{\rm Theorem \ref{unique} is phrased globally, but the same proof holds for the localization to domains $U \subset \w U$ close to a standard Minkowski 
corner $g_{\a_0}$ as discussed in \S 2.1. Thus, uniqueness also holds for the local problem, provided the target data have compact support 
in $\w U$ away from $S\cap U$ and $\cC \cap U$. 
}
\end{remark}

\begin{remark}\label{Dirich}
{\rm We note that Theorem \ref{unique} also holds for Dirichlet boundary data $\cB_{Dir}$, as well as other boundary data arising from a Lagrangian, 
in place of the conformal-mean curvature boundary data $\cB_C$. The proof is the same, using the usual Gibbons-Hawking-York boundary action. 
}
\end{remark}

\section{Completion of Proofs.}

  In this section we complete the proofs of the main results in \S 1. 

\medskip 

 First, the proof of Theorem \ref{ThmI} follows directly from Theorem \ref{globexist} and Theorem \ref{unique}. . 
We turn then to the proof of Theorem \ref{ThmII}. 

\begin{proof}  

    In passing from $\Phi^H$ to $\hat \Phi$ in \eqref{Phi2}, we are dropping the gauge term $V = V_g$ on $S \cup \cC$. By construction, $\hat \Phi$ maps 
into $\cZ$ and so ${\rm Im}D\hat \Phi \subset T(\cZ) \subset T(\hat \cT)$. Recall that $T(\cZ)$ is the subspace of $T(\hat \cT)$ for which the 
linearization of the constraint equations \eqref{Gauss}-\eqref{GC} holds on $S$. In the following, we work on-shell, so that $\Ric_g = 0$ and $V = 0$. 

First it is easy to see that 
$${\rm Im}D\hat \Phi \cap T(\cO_{\Diff_1(M)}) = 0.$$
Namely, for any $h$, $D\hat \Phi(h) = (\Ric'_h, (h_S, K'_h, \nu'_h)_S, (\ring{h}, H'_h)_{\cC}, (\a'_h)_{\Si}) \in T(\hat \cT)$, The group 
$\Diff_1(M)$ acts naturally on the target $\hat \cT$ and a general element of the tangent space to the orbit $T(\cO_{\Diff_1(M)})$ has the form 
$(\d^*Z, 0, \dots, 0)$. An element of the intersection of these two spaces thus satisfies $\Ric'_h = \d^*Z$, for some $h$ and some 
$Z \in T(\Diff_1(M))$. Applying Bianchi operator gives $\b \d^*Z = 0$ and since $Z \in T(\Diff_1(M))$, $Z = 0$ and so $D\hat \Phi(h) = 0$. 

   Thus the main issue is the dense range property in \eqref{transv}. This derives from the dense range of the map $D\Phi^H$ in Theorem \ref{ThmI}.  
 Recall that $T(\cT^H)$ consists of arbitrary data 
$$\tau'= (F,(\g',\k',\nu',V'_S), (\s',\ell', V'_\cC), \a'):= (F,(\iota',\nu',V'_S), (b', V'_\cC), \a').$$
Here for simplicity, we denote the initial data $(\g',\k')_S$ as $\iota'$ and the boundary data $(\s',\ell')_\cC$ as $b'$. We embed $T(\hat \cT)$ into 
$T(\cT^H)$ by setting $V'_S=V'_\cC = 0$, and thus $T(\hat \cT)$ is the subspace of $T(\cT^H)$ given by 
$$T(\hat\cT) = \{ (F,(\iota',\nu',0), (b', 0), \a' \}.$$
The same proof as in Theorem \ref{locexist} shows that $D\Phi^H$ has dense range in $T(\cZ) \subset T(\hat \cT)$, ($D\Phi^H$ maps onto $T(\cZ) \cap \cD$). 
The constraint equations impose a constraint or coupling of the data $(F,(\iota',\nu')$, as discussed at the end of \S 2.2. The dense range property 
above then implies that for a dense set of $\tau'_0 \in T(\cZ) \subset T(\hat \cT)$, there exists a deformation $h \in T_g(Met(M))$, such that 
$D \Phi_g^H (h)= \tau'_0$, i.e.
$$\big(\Ric'_h + \d^*V'_h, (h_S,K'_h, \nu'_h, V'_h), (\ring{h}, H'_h, V'_h), \a'_h\big) = (F,(\iota',\nu',0), (b', 0), \a') = \tau'_0.$$
Also by construction, the data $(\Ric'_h + \d^*V'_h, (h_S, K'_h, \nu'_h))$ satisfies the linearized constraint equations \eqref{Gauss}-\eqref{GC} on $S$. 
As in the proof of Lemma \ref{Gauge-lemma2},  it then follows that $\p_t V'_h = 0$ on $S$. Thus, $V'_h$ vanishes to first order on $S$ and to zero order 
on $\cC$, i.e.~$V'_h \in T(\Diff_1(M))$. Therefore,
\bes
\tau'_0 = \big(\Ric'_h, (h_S,K'_h, \nu'_h,0), (\ring{h}, H'_h, 0), \a'_h\big)+(\d^*V'_h,0,..,0) = D\hat\Phi(h)+(\d^*V'_h,0,..,0)
\ees
This proves the dense range property. 

   The proofs that $D\hat \Phi$ and $D\Phi$ descend to the relevant quotients are then straightforward. 

 \end{proof}

 Next we prove Corollary \ref{ThmIII}. 
 
\begin{proof}
This follows by noting that $T\bE$ is the inverse image $T\bE = (D\Phi)^{-1}(0, \iota', b', \a') = (D\hat \Phi)^{-1}(0, \iota', \nu', b', \a')$ 
and that the kernel of $D\Phi$ is given by the gauge group terms $\d^*X$, $X \in T(\Diff_0(M))$ by Theorem \ref{ThmII}.

\end{proof}

\begin{remark}
{\rm We note that all of the results of this work also hold for the 1-parameter family of boundary conditions considered in \cite{LSW}. These 
are of the form $([g_{\cC}], (dv_{\cC})^p H_{\cC})$, where $dv_{\cC}$ is the volume form density along $\cC$. These boundary conditions 
interpolate between the data $([g_{\cC}], H_{\cC})$ and Dirichlet boundary data $g_{\cC}$ as $p \to 0$ and $p \to \infty$ respectively. 
The proofs remain the same, with only minor modifications. \footnote{We thank Edgar Shaghoulian for discussions on this issue.}

}
\end{remark}

\bibliographystyle{plain}

\end{document}